\documentclass[opre,sglanonrev]{informs4}
\usepackage{hyperref}
\hypersetup{
    colorlinks=true,
    linkcolor=blue}

\usepackage{eqndefns-left}% For checking the display equation width and equation environment definitions %
\RequirePackage{tgtermes}
\RequirePackage{newtxtext}
\RequirePackage{newtxmath}
\RequirePackage{bm}
\RequirePackage{endnotes}
\usepackage[utf8]{inputenc}

\OneAndAHalfSpacedXII 

\usepackage{natbib}
 \bibpunct[, ]{(}{)}{,}{a}{}{,}%
 \def\bibfont{\small}%
\EquationsNumberedThrough    % Default: (1), (2), ...
\TheoremsNumberedThrough     % Preferred (Theorem 1, Lemma 1, Theorem 2)
\ECRepeatTheorems  %  

\usepackage{tabularx}
\usepackage{multirow}
\usepackage{url}
\newif\ifrepeattheo
\repeattheotrue
\usepackage{anyfontsize}
\usepackage{thm-restate}
\usepackage{enumitem}
\usepackage[capitalise]{cleveref}

\newcommand{\pointtoproof}[1]{$\longrightarrow$ See page \pageref{#1} for the proof.}
\begin{document}

\TITLE{Convex duality contracts for production-grade mathematical optimization}

\ARTICLEAUTHORS{%

\AUTHOR{Juan Pablo Vielma,\textsuperscript{a} Ross Anderson,\textsuperscript{b} Joey Huchette\textsuperscript{c}}
\AFF{Google Research,
\textsuperscript{a}
 \EMAIL{jvielma@google.com};
 \textsuperscript{b}
 \EMAIL{rander@google.com};
 \textsuperscript{c}
 \EMAIL{jhuchette@google.com};
}

} 

\ABSTRACT{
Deploying mathematical optimization in autonomous production systems requires precise contracts for objects returned by an optimization solver. Unfortunately, conventions on dual solution and infeasibility certificates (rays) vary widely across solvers and classes of problems. This paper presents the theoretical framework used by MathOpt (a domain-specific language developed and used at Google) to unify these notions. We propose an abstract primal-dual pair based on a simplified Fenchel duality scheme that allows for the mechanical derivation of dual problems and associated contracts for all classes of problems currently supported by MathOpt (including those with linear and quadratic objectives plus linear, conic, quadratic, and two-sided linear constraints). We also show how these contracts can improve clarity of complementary-slackness based optimality conditions for certain classes of problems.
}
\maketitle
\vfill
\pagebreak
\section{Introduction}

A domain-specific language (DSL) for mathematical optimization provides an interface through which users can model and solve complicated decision-making problems. A DSL provides an easy-to-use modeling interface and serves as a bridge between this higher-level description of the problem and an underlying optimization solver that is used to produce an answer. The value of this bridge grows as the range of problems you can model, and the number of solvers you can use to produce an answer, grows. However, this growth makes abstraction and unification harder.

MathOpt \citep{MathOpt} is an open-source DSL for mathematical optimization developed by, and used extensively at, Google. Because MathOpt is used  widely in autonomous production systems, it needs to provide \emph{contracts} that are (i) precise and complete, (ii) easy to follow by users without a deep theoretical optimization background, and (iii) amenable to standard software engineering practices like unit testing (e.g. \cite{kanewala2014testing,winters2020software}).

Convex optimization duality is an extremely versatile tool to certify optimality, unboundedness and infeasibility of an optimization problem and is at the core of most convex optimization solvers.  However, there are a plethora of possible expressions for duality, depending on the particular flavor of optimization problem (e.g., linear, quadratic, conic, etc.) that you are considering. Indeed, even linear programming duality depends on the particular ``standard form'' you use (see Section~\ref{LP:properties:section}), meaning that different solvers can provide different duality contracts, even for the same class of problems. As a result, duality was the thorniest part of the MathOpt contract to get right.

This work presents the duality framework developed and used within MathOpt. This takes the form of an abstract primal-dual optimization pair, drawing inspiration from recent work \citep{rockafellar2023augmented,roos2020universal}. The main goals of this framework are to:
\begin{enumerate}[label=G\arabic*]
    \item\label{goal:precise} \emph{Have precise, mechanical and modular contracts:} The framework provides a clear, formal mathematical contract that can be efficiently implemented in code.
        \item\label{goal:expressive} \emph{Be fully expressive, generic and extensible:} The framework provides contracts for all dual problem properties and uses (e.g. certifying optimality, unboundedness and infeasibility), and they are applicable for all objective functions and constraints currently supported by MathOpt. 
    \item\label{goal:reduce}  \emph{Reduce to existing duality frameworks}: The framework reduces to standard textbook and solver-specific duality derivations through simple recipes. In particular, all duality information provided by MathOpt's supported solvers can be communicated to the user via the framework. 
\end{enumerate}

To the best of our knowledge, ours is the first duality framework to offer all of this together.

The primal problem for MathOpt's duality framework is  given by
\begin{equation}\label{primalproblem}
\operatorname{opt}_P:=\inf_{x\in \mathbb{R}^J} \; c\cdot x+\sum_{i\in I} g_i\left(F^i(x)\right)
\end{equation}
where $J$ and $I$ are finite sets\footnote{We use abstract sets instead of ordered sets like $\{1,\ldots,n\}$ to keep the notation concise and mirror what is supported in MathOpt.}, $c\in\mathbb{R}^J$ and for each $i\in I$, $M_i$ is a finite set, $F^i:\mathbb{R}^J\to \mathbb{R}^{M_i}$ is a continuously differentiable function and $g_i:\mathbb{R}^{M_i} \to {\mathbb{R}\cup\{+\infty\}}$ is a proper closed convex function.

MathOpt's dual for \eqref{primalproblem} is a simplification of Fenchel's dual that exploits the splitting into $F^i$ and $g_i$ components (see Section~\ref{connectiontofenchelsection})  and is given by
\begin{subequations}\label{dualproblem}
\begin{alignat}{3}
\operatorname{opt}_D:=&\sup_{y\in \mathbb{R}^{M}, w\in \mathbb{R}^{I\times J}}&   -\sum_{i \in I}   F^i_\nabla\left(-y^i, w^i\right)&+g^*_i\left(-y^i\right)\label{dualproblem:obj}\\
&s.t.\notag&&\\
&&\sum_{i\in I}\nabla F^i\left(w^i\right)^T y^i&= c,&\label{stationarity}
\end{alignat}
\end{subequations}
where  $M:=\bigtimes_{i\in I}M_i$, $\nabla F^i:\mathbb{R}^J\to \mathbb{R}^{M_i\times J}$ is the \emph{Jacobian} of $F^i$ given by 
$(
\nabla F^i(x))_{k,j}=
\frac{\partial F^i_k}{\partial x_j}(x)$ for all  $k\in M_i$ $j\in J$,
$g^*_i:\mathbb{R}^{M_i}\to {\mathbb{R}\cup\{+\infty\}}$  is the convex conjugate of $g_i$ given by 
$g^*_i\left(\bar{v}\right):=\sup_{z\in \mathbb{R}^{M_i}}\left\{\bar{v}\cdot z-g\left(z\right) \right\}$
and $F^i_\nabla:\mathbb{R}^{M_i}\times\mathbb{R}^J\to \mathbb{R}\cup\{+\infty\}$ is given by $
F^i_\nabla\left(y^i, w^i\right):=\left(\nabla F^i\left(w^i\right)w^i -  F^i\left(w^i\right)\right)\cdot y^i$. Table~\ref{initialfunctions} details the functions needed to construct  \eqref{primalproblem}/\eqref{dualproblem} for all objective functions and constraints supported by MathOpt.

\newcolumntype{T}{>{\hsize=.18\hsize}X}
\newcolumntype{g}{>{\hsize=.10\hsize}X}
\newcolumntype{L}{>{\hsize=.17\hsize}X}
\newcolumntype{N}{>{\hsize=.12\hsize}X}
\newcolumntype{M}{>{\hsize=.23\hsize}X}
\newcolumntype{S}{>{\hsize=.20\hsize}X}
\begin{table}[htp!]
\centering
 \begin{tabularx}{\textwidth}{TgLNMS}
& $g_i(z)$ & $F_i(x)$ & $\nabla F\left(w^i\right)$ & $F^i_\nabla\left(-y^i, w^i\right)$&$g_i^*\left(-y^i\right)$\\
  \hline
 Linear ineq. &$\delta_{\mathbb{R}^{M_i}_+}(z)$ & $Ax-b$ & $A$ & $-b\cdot y^i$&$\delta_{\mathbb{R}^{M_i}_-}\left(-y^i\right)$\\
  Conic constr.&  $\delta_{K}(z)$ & $Ax-b$ & $A$ & $-b\cdot y^i$&$\delta_{K^\circ}\left(-y^i\right)$\\
  Two-sided ineq. &  $\delta_{[l,u]}(z)$ & $a\cdot x$ & $a$ & $0$&$\max\left\{-l y^i,-u y^i\right\}$
  \\
  Quadratic obj. &  $z$ & $x^T Qx$ & $2Qw^i$ & $ -y^i\left(w^i\right)^T Qw^i$&$\delta_{\{1\}}(-y^i)$
    \\
  Quadratic ineq. & $\delta_{\mathbb{R}_-}(z)$ &$x^T Qx+q\cdot x-u$ & $2Qw^i+q$ & $-y^i \left(\left(w^i\right)^T Qw^i+u\right)$&$\delta_{\mathbb{R}_-}(-y^i)$
 \end{tabularx}
\caption{Functions associated with objective functions and constraints supported by MathOpt: linear inequalities $Ax\geq b$, conic constraints $Ax-b\in K$ (currently only for LP and SOC cones), two-sided linear inequalities $l\leq a\cdot x\leq u$, convex quadratic objectives $x^T Q x$ and convex quadratic inequalities $x^T Q x+q\cdot x\leq u$, where $A\in \mathbb{R}^{M_i\times J}$, $b\in \mathbb{R}^{M_i}$, $K\subseteq \mathbb{R}^{M_i}$ is a closed convex cone,  $K^\circ:=\{y\in \mathbb{R}^{M_i}\,:\,y\cdot x\leq 0\;\forall x\in K\}$ is the polar cone of $K$  , $l,u\in \mathbb{R}$ are such that $l\leq u$, $Q\in \mathbb{R}^{J\times J}$ is  positive semi-definite, $q\in \mathbb{R}^J$ and $\delta_{S}:\mathbb{R}^{L}\to \mathbb{R}\cup\{+\infty\}$  is the indicator of a closed convex set $S\subseteq \mathbb{R}^{L}$ (i.e. $\delta_{S}(z):=0$ if $z\in S$, and $\delta_{S}(z):=+\infty$ otherwise)}\label{initialfunctions}
\end{table}

With regards to goal~\ref{goal:precise}, dual problem \eqref{dualproblem} is clearly precise and modular. For instance, there is a group of variables $(y^i,w^i)$ for each module $i\in I$ from primal problem \eqref{primalproblem} and they are only connected through equation \eqref{stationarity}. Dual problem \eqref{dualproblem} can also be mechanically constructed if $\nabla F^i$ and $g_i^*$ can be mechanically constructed. Such assumption on $\nabla F^i$ is fairly standard and holds if $ F^i$ is affine or quadratic. The assumption on $g_i^*$ is less standard, but is not too restrictive and holds for all classes of objectives and constraints currently supported by MathOpt as detailed in  Table~\ref{initialfunctions}. 

With regards to goal \ref{goal:expressive} we clearly have that the framework is generic, extensible, and covers all objectives and constraints currently supported by MathOpt (see Table~\ref{initialfunctions}). Assuming weak duality holds for \eqref{primalproblem}/\eqref{dualproblem} (which we formalize in \cref{weakduality}), we can infer contracts for dual solutions as optimality certificates. We more formally describe these and other contracts in Section~\ref{LP:properties:section}.

With regards to goal \ref{goal:reduce}, we can use the functions in Table~\ref{initialfunctions} to recover standard dual forms as detailed in  Section~\ref{secificssection} and Appendix~\ref{max:sec} (Section~\ref{secificssection} also includes formal propositions for the derivations in Table~\ref{initialfunctions}). In particular, note that use of $-y^i$ in \eqref{dualproblem:obj} is to ensure we align with the standard sign conventions for conic and linear programming duality (see Sections~\ref{conic_constraints} and \ref{conicsectionmax}). 

Finally, two observations are important about the structure of \eqref{primalproblem}/\eqref{dualproblem}. First, note that all functions $g_i$ in Table~\ref{initialfunctions} are either indicator functions or the trivial identity function. We do not explicitly restrict functions $g_i$ to these cases as the more general case slightly simplifies the abstract descriptions and does not impose a practical burden (assuming $g_i^*$ has a mechanical construction as described above). In addition, the more general case allows extending   Table~\ref{initialfunctions} to some solver-supported non-differentiable objectives such as convex piecewise-linear functions (e.g. function $\max\left\{-l y^i,-u y^i\right\}$ that appears in the objective of \eqref{dualproblem} for two-sided linear inequalities\footnote{In particular, this makes  \eqref{dualproblem} a special case of \eqref{primalproblem}, but making the primal-dual pair fully symmetric is not part of our goals.}). 

The second observation is the apparent non-convexity in \eqref{dualproblem} arising from the product of the $w^i$ and $y^i$ variables. This is just an artifact of writing a mechanical description of the more evidently convex problem behind \eqref{dualproblem} and also appears in Wolfe's dual (see Section~\ref{lagrangiansection}) and when writing concrete formulations of Fenchel's dual (e.g. \cite{roos2020universal}). As shown in Section~\ref{secificssection} this apparent non-convexity can be eliminated for all functions in Table~\ref{initialfunctions}.

\subsection{Duality contracts in Linear Programming}\label{LP:properties:section}

The contracts we want for pair \eqref{primalproblem}/\eqref{dualproblem} are those available for linear programming (LP) in their full generality as we now detail. 
We begin with the ``textbook'' inequality-based LP given by
\begin{subequations}\label{basic:lp}
    \begin{alignat}{3}
\operatorname{opt}_{P}:=&\min_{{x}\in \mathbb{R}^J}& 
\quad {c}
\cdot &{x}\quad &\\
&s.t.\notag&&\\
&&
{A^i}{
x}&
\geq {b^i},
\quad 
\forall i\in I,&\label{basic:lp:constraints}
\end{alignat}
\end{subequations}
where $I$ and $J$ are finite sets, $c\in\mathbb{R}^J$, and for each $i\in I$, $M_i$ is a finite set, $A^i
\in \mathbb{R}^{M_i\times J}$, $b^i\in\mathbb{R}^{M_i}$, and the inequality for elements of $\mathbb{R}^{M_i}$ is taken component-wise. 

The dual of  \eqref{basic:lp} is given by 
\begin{subequations}\label{basic:dual:lp}
    \begin{alignat}{3}
\operatorname{opt}_{D}:=&\max_{{y}\in \mathbb{R}^M}& 
\quad 
\sum_{i
\in I}{b^i}
\cdot &{y^i}\quad &\\
&s.t.\notag&&\\
&&
\sum_{i
\in I}\left({A^i}\right)^T{
y^i}&
= {c},&\label{basic:dual:lp:stationarity}\\
&&
{
y^i}&
\geq 0\quad \forall i\in I.\label{basic:dual:lp:nn}
\end{alignat}
\end{subequations}

The first class of properties and contracts concerns solutions and can be specified as follows. For any primal solutions $\bar{x}\in \mathbb{R}^J$ and dual solution $\bar{y}\in \mathbb{R}^M$ have the following:
\begin{itemize}
    \item \textbf{Clear notion of feasibility:} Satisfying modular constraints \eqref{basic:lp:constraints} for $\bar{x}$, and satisfying modular constraints \eqref{basic:dual:lp:nn} plus the simple equation \eqref{basic:dual:lp:stationarity} for $\bar{y}$.
     \item \textbf{Clear notions of objective value even for infeasible solutions:} $c\cdot \bar{x}$ and $\sum_{i\in I}{b^i}\cdot {\bar{y}^i}$.
     \item \textbf{Weak duality:}\footnote{We did not include strong duality as it is less relevant when \emph{receiving} solutions from a solver (e.g. strong duality enables a solver to return an optimal solution, but it could still fail for numerical reasons or because of a user defined time limit).
     } $c\cdot \bar{x}\geq \sum_{i\in I}{b^i}\cdot {\bar{y}^i}$ for any feasible $\bar{x}$ and $\bar{y}$.
     \item \textbf{Optimality conditions based on modular complementary slackness (CS):} If  $\bar{x}$ and $\bar{y}$ satisfy
     \begin{equation}\label{basic:lp:cs}
         \bar{y}^i\cdot (A^i \bar{x} -b^i)=0\quad\forall i \in I,
     \end{equation}
     then $c\cdot \bar{x}=\sum_{i\in I}{b^i}\cdot {\bar{y}^i}$ even if either or both solutions are infeasible, and for feasible  $\bar{x}$ and $\bar{y}$ condition \eqref{basic:lp:cs} is equivalent to optimality of both solutions.
\end{itemize}

We included the fact that CS implies equal objectives even in the absence of feasibility because of its relevance when receiving sub-optimal basic solutions from simplex-based LP solvers. In particular, the notions of an objective value for infeasible solutions is critical for this CS property. Similarly, low complexity (in description and computational cost) notions of feasibility are critical to exploit weak duality to obtain primal and dual bounds (e.g. from sub-optimal solutions from primal-simplex and dual-simplex respectively). For LP, these notions of feasibility and objective values for feasible and infeasible solutions are straightforward, but they are less clear for \eqref{primalproblem}/\eqref{dualproblem}. In Sections~\ref{primal:sec} and \ref{dualsection} we  show how these notions (and the associated CS properties) can be extended to \eqref{primalproblem}/\eqref{dualproblem}. In Section~\ref{connectiontofenchelsection} we discuss how dual problem \eqref{dualproblem} reduces the complexity of checking dual feasibility as compared to a direct description of Fenchel's dual for \eqref{primalproblem}.

The second class of properties and contracts concerns rays and can be specified as follows:
\begin{itemize}
    \item \textbf{Primal ray:} A primal ray $\hat{x}\in \mathbb{R}^J$ satisfies \eqref{basic:lp:constraints} with $b^i=0$ for all $i\in I$ and $c\cdot\hat{x}<0$.
       \item \textbf{Dual ray:} A dual ray $\hat{y}\in \mathbb{R}^M$ satisfies \eqref{basic:dual:lp:stationarity} with $c=0$, \eqref{basic:dual:lp:nn} and $\sum_{i\in I}{b^i}\cdot {\hat{y}^i}>0$.       
\item \textbf{Unbounded primal:} \eqref{basic:lp} is unbounded if there is a primal feasible solution and a primal ray.
\item \textbf{Unbounded dual:}  \eqref{basic:dual:lp} is unbounded if there is a dual feasible solution and a dual ray.
\item \textbf{Infeasible Primal:} \eqref{basic:lp} is infeasible if there is a dual ray.
\item \textbf{Infeasible Dual:} \eqref{basic:dual:lp} is infeasible if there is a primal ray.
\end{itemize}
The first role of rays is to certify  unboundedness of their respective problems together with an respective solution. The second role is to independently certify unboundedness of the other problem. 
In particular, both problems could have their infeasibilities certified by rays at the same time.

Now, LP form \eqref{basic:lp} considers only $\geq$-based inequalities, but we certainly can consider $\leq$-based inequalities or $=$-based inequalities (i.e. equations). Including these additional inequalities just changes the sign-constraints \eqref{basic:dual:lp:nn} in the dual (for $\leq$-based inequalities the dual variables are non-positive and for equations they are sign-unrestricted). Less obvious is the fact that we can also consider inequalities that have lower and upper bounds that are not necessarily identical (i.e. are not necessarily equations), which we refer to as \emph{two-sided linear inequalities}. For instance, let $J$ and $M$ be finite sets,  $c\in\mathbb{R}^J$, $A
\in \mathbb{R}^{M\times J}$, $b\in\mathbb{R}^{M}$ and consider the LP given by 
\begin{subequations}\label{standard:ub:lp}
    \begin{alignat}{3}
\operatorname{opt}_{P}:=&\min_{{x}\in \mathbb{R}^J}& 
\quad {c}
\cdot &{x}\quad &\\
&s.t.\notag&&\\
&&
{A}{
x}
&= {b},\label{standard:ub:lp:eq}\\
&&l_j\leq x_j&\leq u_j\quad \forall j\in J.\label{standard:ub:lp:variable:bounds}
\end{alignat}
\end{subequations}
Splitting \eqref{standard:ub:lp} into $l_j\leq x_j$ and $x_j\leq u_j$ we can consider the problem as having three modules and associated sets of dual variables: one sign-unrestricted group for \eqref{standard:ub:lp:eq} and two sign-constrained groups for \eqref{standard:ub:lp:variable:bounds}. However, simplex-based solvers for \eqref{standard:ub:lp} return only two groups of dual variables. The first group of dual variables $\pi\in \mathbb{R}^{M}$ is the expected set of unrestricted dual variables for \eqref{standard:ub:lp:eq}. The second set of dual variables $r\in \mathbb{R}^{J}$  associated to \eqref{standard:ub:lp:variable:bounds} is also unrestricted and their elements $r_j$ are denoted the \emph{reduced costs} of variables $x_j$. We can deduce the sign of $r_j$ for basic solutions through rules that are natural in the context of the simplex algorithm (e.g. if basic pair $\bar{x}$/$(\bar{y},\bar{r})$ has $\bar{x}_j=l_j$, then $\bar{r}_j\geq 0$). However, these rules could be less intuitive for users unfamiliar with the simplex algorithm. Fortunately, using monotropic programming duality \citep{rockafellar1981monotropic} we can construct the following dual of \eqref{basic:lp} (e.g \cite{fourer94}) that suggests clearer contracts are possible.
\begin{subequations}\label{standard:ub:dual:lp}
    \begin{alignat}{3}
\operatorname{opt}_{D}:=&\max_{\pi\in \mathbb{R}^{M}, r\in  \mathbb{R}^{J}}& 
\quad 
{b}
\cdot  \pi +\sum_{j\in J} &\min\{l_j r_j, u_j r_j\}&\\
&s.t.\notag&&\\
&&
A^T \pi + r&
= {c}.&
\end{alignat}
\end{subequations}
As shown in Section~\ref{twosideddec}, \eqref{standard:ub:dual:lp} is also a special case of MathOpt dual problem  \eqref{dualproblem}. 
In this work we will show how all these concepts and associated contracts extend to our proposed duality framework. 

\subsection{Existing dual forms}\label{existing:sec}

\subsubsection{Conic dual and its extensions}\label{conic:ext:sec}

One dual framework that satisfies most of MathOpt's requirements is  
conic duality \citep{ben2001lectures}. One thing missing is that including two-sided linear inequalities, quadratic objectives and quadratic constraints requires  reformulations into conic forms. These reformulations are simple and apply to a wider class of constraints (e.g. \cite{aps2020mosek}), which is one of the reasons other DSLs like JuMP \citep{lubin2023jump} adopted the conic dual. However, the associated reformulations have at least two potential complications, which are  partially mitigated in JuMP through a constraint/solution ``bridge'' system  \citep{legat2022mathoptinterface}.

The first complication concerns the sign-unconstrained reduced costs for variables with upper and lower bounds. As discussed in Section~\ref{LP:properties:section} such a \emph{two-sided variable bound} could be split  into two \emph{one-sided variable bounds}, but the sign rules associated to this definition could be confusing. 
Fortunately, existing extensions of LP and conic duality (e.g. \cite{fourer94,karimi2020primal,karimi2024domain,rockafellar1981monotropic}) show how non-conic convex sets (e.g. intervals as in the case of two-sided variable bounds and linear inequalities) can be directly considered. These approaches fit within our proposed framework (see Sections~\ref{twosideddec} and \ref{dds:section}). 

The second complication concerns quadratic objectives and quadratic constraints. While the associated reformulations are simple, the resulting dual solutions can be confusing to non-expert users. For instance, the conic reformulation of a quadratic objective adds a new conic constraint and the presence of dual variables for that constraint can be surprising. Fortunately, conic duality can also be extended to consider quadratic objectives  \citep{vandenberghe2010cvxopt} in a way that also fits in our proposed framework (see Section~\ref{quadraticobjsec}). In contrast, quadratic constraints are more complicated and current approaches are closer to an extension of Lagrangian duality to conic constraints than an extension of conic duality to quadratic constraints so we will discuss them in detail in Section~\ref{lag:dual:sec}. 

Another thing missing  is extensibility to consider constraints and objectives without known conic reformulations. This includes non-linear objective functions and \emph{non-linear inequalities} (i.e. constraints defined as level sets of non-linear functions). Existing extensions for this case are closer to extensions of Lagrangian duality to conic constraints so we discuss them in detail in Section~\ref{lag:dual:sec}. 

\subsubsection{Lagrangian duality and its extensions}\label{lag:dual:sec}

One dual form that is extremely versatile is Lagrangian duality (e.g. \cite{boyd2004convex}) and particularly its extensions to consider non-differential functions. Examples of such extensions include those for problems that combine non-linear inequalities based on differentiable functions and second-order-cone constraints \cite{knitro} and for problems that combine quadratic objectives, quadratic inequalities and conic constraints \citep{MOIDual}. More generally, the generalized nonlinear programming (GNLP) framework of \cite{rockafellar2023augmented} details an extension that covers all convex optimization problems. 

One issue with Lagrangian duality is that the associated dual optimization problem is not of the same form as the primal optimization problem because it combines two optimization senses (e.g. if the primal has a $\min$ sense the Lagrangian dual has a $\max \min$ sense). It is well known that one of the senses of the Lagrangian dual can be eliminated to obtain a single sense version often denoted Wolfe's dual \citep{dorn1960duality,wolfe1961duality}. For instance, this is used in \cite{MOIDual} to show that the Lagrangian dual for quadratic objective  + conic constraints simplifies to the conic dual extension from  \cite{vandenberghe2010cvxopt}. However, as we detail in Section~\ref{lagrangiansection}, even after this simplification the dual does not have a modular structure because it contains the original primal variables and these directly interact with the dual variables associated to all constraints. Nonetheless, as we also detail in Section~\ref{lagrangiansection}, our proposed framework can be interpreted as a variant of the GNLP Lagrangian dual from \cite{rockafellar2023augmented}. Hence our framework also covers all convex optimization problems.  

Finally, closely related are the optimality conditions for two-sided linear and non-linear constraints considered in \cite{knitrotermination}. As we detail in Section~\ref{knitro:cs:section}, these can be interpreted as coming from the GNLP Lagrangian dual and hence are similar to our optimality conditions related to complementary slackness. However, as we also detail in Section~\ref{knitro:cs:section}, the conditions from \cite{knitrotermination} have similar issues around a precise and succinct description of dual variable sign-rules as those for reduced costs for variables with lower and upper bounds as noted in Section~\ref{LP:properties:section}.

\subsubsection{Fenchel dual}\label{fenchel:dual:sec}

Fenchel's duality theorem can be used to derive abstract dual forms for any convex optimization problem. In addition, as noted in
 \cite{roos2020universal}, these abstract dual forms can often be made concrete using calculus rules for the convex conjugate of a convex function. As detailed in Section~\ref{connectiontofenchelsection}, the only issue with this dual form is that its construction and detailed description of its properties is only mechanical after it is reduced to a concrete form for specific classes of problems. As we also detail in Section~\ref{connectiontofenchelsection}, our proposed framework can be interpreted as an intermediate simplification of the abstract and concrete versions of the dual in \cite{roos2020universal} that allows a generic, but mechanical description of all the properties we seek.

\subsection{Contributions and outline}

A first contribution of MathOpt's duality framework is a simplification of standard Fenchel duality framework to enable mechanical contracts for a wide range of constraints. As detailed in Section~\ref{connectiontofenchelsection} this simplification is achieved by exploiting the composite structure  (e.g. \cite{rockafellar2023augmented}) in primal problem \eqref{primalproblem}. The key observation here is that the Fenchel dual of  \eqref{primalproblem} includes functions $\left(g_i\circ F^i\right)^*$ (where $g_i\circ F^i(x)=g_i(F^i(x))$), which can be hard to describe mechanically as a function of $g_i\circ F^i$. However, we show that an appropriate tight approximation of $\left(g_i\circ F^i\right)^*$ can be mechanically described as a function of $\nabla F^i$ and $g_i^*$.

The second contribution is to extend the mechanical construction of dual \eqref{dualproblem} to mechanical contracts for all the associated uses of duality as described in Section~\ref{LP:properties:section}. In particular, these contracts can significantly improve clarity of complementary-slackness based optimality conditions for certain classes of constraints (e.g. see Section~\ref{knitro:cs:section}).

The remainder of this work is organized as follows. First, in Section~\ref{notation:section} we give some notation. In Section~\ref{primal:sec} we introduce definitions and contracts for primal solutions. We also detail some baseline assumptions for the results in the paper. Then, in Section~\ref{dualsection} we introduce definitions and contracts for dual solutions. In particular, \cref{weakduality} establishes weak duality for \eqref{primalproblem}/\eqref{dualproblem} and \cref{csoptimalitycorollary} establishes complementary-slackness based optimality conditions for \eqref{primalproblem}/\eqref{dualproblem}. Section~\ref{rayssection} introduces contracts for primal and dual rays and their role as unboundedness and infeasibility certificates. Section~\ref{primal:sec}, \ref{dualsection} and \ref{rayssection} present generic or abstract descriptions so in Section~\ref{secificssection} we make everything concrete for the objectives and constraints currently supported by MathOpt. In particular, Section~\ref{secificssection} shows how all associated contracts become mechanical and match standard dual forms when applicable. Finally, Section~\ref{connectiontodualssection} provides a comparison to other duality frameworks, especially those from  \cite{rockafellar2023augmented} and \cite{roos2020universal}. All proofs are relegated to Section~\ref{proofs:sec}. Throughout the paper we restrict to primal minimization problems for succinctness. The associated changes for primal maximization problems are simple and also mechanical and are detailed in Appendix~\ref{max:sec}.

\subsection{Notation}\label{notation:section}

For any finite set $L$ and $S\subseteq \mathbb{R}^L$ we let
$\delta_{S}:\mathbb{R}^{L}\to \mathbb{R}\cup\{+\infty\}$  be  the \emph{indicator function} of $S$ (i.e. $\delta_{S}(z):=0$ if $z\in S$, and $\delta_{S}(z):=+\infty$ otherwise) and   $\sigma_S:\mathbb{R}^L\to \mathbb{R}\cup\{+\infty\}$ be the \emph{support function} of $S$ (i.e. $\sigma_S(d):=\sup\{d\cdot z\,:\, z\in S\}$). For any $h:\mathbb{R}^L\to \mathbb{R}\cup\{+\infty\}$ we let $\operatorname{dom}(h):=\{z\in \mathbb{R}^{L}\,:\, h(z)<+\infty\}$ be the \emph{domain} of $h$ and $\operatorname{rge}(h):=\{h(z)\in \mathbb{R}\cup\{+\infty\}\,:\, z\in \mathbb{R}^L\}$ be the range of $h$. Finally, for any $F^i:\mathbb{R}^J\to \mathbb{R}^{M_i}$  and $g_i:\mathbb{R}^{M_i} \to {\mathbb{R}\cup\{+\infty\}}$  we let $g_i\circ F^i:\mathbb{R}^J\to \mathbb{R}\cup\{+\infty\}$ be the function given by $g_i\circ F^i(x)=g_i\left(F^i\left(x\right)\right)$.

\section{Primal problem definitions and assumptions}\label{primal:sec}

Because functions $g_i$ can take the value $+\infty$, they implicitly encode the constraints of \eqref{primalproblem}. 

\begin{definition}[Primal Feasibility]
A primal solution $\bar{x}\in \mathbb{R}^J$ is \emph{feasible} if $g_i(F^i(\bar{x}))<+\infty$  $\forall i\in I$. Primal problem \eqref{primalproblem} is feasible if it has a feasible primal solution. 
\end{definition}

When the only infinite valued functions are indicator functions (such as in Table~\ref{initialfunctions}) it makes sense to define an objective function that is finite valued also for infeasible solutions as follows. 

\begin{definition}[Primal objective value]
For any $h:\mathbb{R}^{L} \to {\mathbb{R}\cup\{+\infty\}}$ we let $h^0:\mathbb{R}^{L} \to {\mathbb{R}\cup\{+\infty\}}$ be the identically zero function if $h$ is an indicator function and $h$ otherwise. That is
\[
h^0(x):=\begin{cases}0 & \operatorname{rge}(h)\subseteq \{0,+\infty\}\\ h(x)&\text{o.w.}\end{cases}
\]
The objective value of any (feasible or infeasible) primal solution $\bar{x}\in \mathbb{R}^J$ is
\begin{equation}\label{primalobjmin}\operatorname{obj_P}(\bar{x}):=c\cdot \bar{x}+\sum_{i\in I} g_i^0\left(F^i\left(\bar{x}\right)\right).\end{equation}
\end{definition}

Finally, throughout the paper we make the following assumptions on functions $g_i$ and $F^i$.

\begin{assumption}\label{assumption}
For each $i\in I$,
\begin{subequations}\label{assumptioneq}
\begin{align}
\label{assumptioneq:gofproper}
&\exists \bar{x}^i\in \mathbb{R}^J \text{ s.t. }g_i\left(F^i\left(\bar{x}^i\right)\right)<+\infty,\text{ and}\\
\label{assumptioneq:fdiff}
    &F^i \text{ is continuously differentiable},\\ 
    \label{assumptioneq:gproper}
    &g_i 
    \text{ is a proper, closed and convex function},\\
    \label{assumptioneq:goffullyconvex}
    &g_i\left(F^i(x)+u\right) \text{ is jointly convex in $(x,u)$}.
\end{align}
\end{subequations}
\end{assumption}
Assumption \eqref{assumptioneq:gofproper}  excludes functions that independently cause infeasibility. The rest of the assumptions are always satisfied by the functions in Table~\ref{initialfunctions}. However, note that while \eqref{assumptioneq:goffullyconvex} is often satisfied for convex optimization problems, it is slightly stronger than convexity of $g_i\circ F^i$
\footnote{e.g. if $F^i(x)=x^2$ and $g_i(z)=\delta_{[0,1]}$, then  $g_i\circ F^i$ is convex, but  \eqref{assumptioneq:goffullyconvex} is not satisfied}.

\section{Dual problem definitions and properties}\label{dualsection}

The only explicit constraints in \eqref{dualproblem} are equations \eqref{stationarity}, which are often denoted \emph{stationarity} conditions. However,  because functions  $g_i^*$ in objective \eqref{dualproblem:obj} can also take the value $+\infty$, they implicitly encode additional constraints for \eqref{dualproblem}. Following the convention in KKT conditions, we partially separate these additional module-wise constraints from the joint stationarity condition \eqref{stationarity} (i.e. we do not include the stationarity condition when considering dual feasibility for a single $i\in I$).

\begin{definition}[Dual Feasibility]
For each $i\in I$  a (module-wise) dual solution $(\bar{y}^i,\bar{w}^i)\in \mathbb{R}^{M_i}\times \mathbb{R}^{J}$ is \emph{(module-wise) feasible} if and only if $g_i^*(-\bar{y}^i)<+\infty$.

A dual solution $(\bar{y},\bar{w})\in \mathbb{R}^{M}\times \mathbb{R}^{I\times J}$ is \emph{feasible} if and only if $(\bar{y}^i,\bar{w}^i)$ is feasible for each $i \in I$ and $(\bar{y},\bar{w})$ satisfies stationarity condition \eqref{stationarity}. 

Dual problem \eqref{dualproblem}  is feasible if and only if it has a feasible dual solution. 

\end{definition}
We also define a dual objective function  that is finite when a solution is infeasible only due to an indicator function $g_i^*$ in the dual objective function \eqref{dualproblem:obj} .

\begin{definition}[Dual objective function and value]\label{dual:obj:def}
Let $\mathcal{O}_{F^i,g_i}\left(y^i, w^i\right):=F^i_\nabla\left(y^i, w^i\right)+g^*_i\left(y^i\right)$,  so the dual objective function \eqref{dualproblem:obj} is equal to $-\sum_{i \in I} \mathcal{O}_{F^i,g_i}(-{y}^i, {w}^i)$.

The objective value of any (feasible or infeasible) dual solution $(\bar{y},\bar{w})\in \mathbb{R}^{M}\times \mathbb{R}^{I\times J}$ is
\begin{equation}\label{dualobjectivemin}
    \operatorname{obj_D}(\bar{y},\bar{w}):=-\sum_{i \in I} \mathcal{O}_{F^i,g_i}^0(-\bar{y}^i, \bar{w}^i),
\end{equation}
where $\mathcal{O}_{F^i,g_i}^0(\bar{y}^i, \bar{w}^i):=F^i_\nabla\left(\bar{y}^i, \bar{w}^i\right)+\left(g^*\right)^0\left(\bar{y}^i\right)$.

\end{definition}

Even without qualification conditions beyond \eqref{assumptioneq} we have the following weak duality result.
\begin{restatable}[Weak Duality]{theorem}{weakduality}\label{weakduality}
If $\bar{x}\in \mathbb{R}^J$ and $(\bar{y},\bar{w})\in \mathbb{R}^{M}\times \mathbb{R}^{I\times J}$  are such that $(\bar{y},\bar{w})$ satisfies stationarity condition \eqref{stationarity}, then
\begin{equation}\label{weakdualityone}c\cdot \bar{x}+\sum_{i\in I} g_i\left(F^i\left(\bar{x}\right)\right)\geq
-\sum_{i \in I}   \mathcal{O}_{F^i,g_i}(-\bar{y}^i, \bar{w}^i).
\end{equation}
In particular, if $\bar{x}$ and $(\bar{y},\bar{w})$  are feasible, then  
\begin{equation}\label{weakdualitytwo}\operatorname{obj_P}(\bar{x})\geq \operatorname{obj_D}(\bar{y},\bar{w}).\end{equation}
\end{restatable}
\pointtoproof{weakduality:proof}

Guaranteeing existence of a primal dual pair that satisfies the weak duality inequality as equality does require additional constraint qualification conditions. However, the following  module-wise optimality conditions do not require such constraint qualifications (see first observation below).

\begin{restatable}[Optimality Conditions]{theorem}{optimalityconditions}\label{optimalityconditions}
Let $\bar{x}\in \mathbb{R}^J$ and $(\bar{y},\bar{w})\in \mathbb{R}^{M}\times \mathbb{R}^{I\times J}$  be such that $(\bar{y},\bar{w})$ satisfies stationarity condition  \eqref{stationarity}. Then, the following conditions are equivalent
\begin{itemize}
    \item Primal-dual pair $(\bar{x},(\bar{y},\bar{w}))$ is  optimal with
    $c\cdot \bar{x}+\sum_{i\in I} g_i\left(F^i\left(\bar{x}\right)\right)= -\sum_{i \in I}   \mathcal{O}_{F^i,g_i}(-\bar{y}^i, \bar{w}^i)$.
    \item For all $i\in I$, 
    \begin{equation}\label{optimaliticonditioneq}g_i\left(F^i\left(\bar{x}\right)\right)+\mathcal{O}_{F^i,g_i}(-\bar{y}^i, \bar{w}^i)=-\left(\bar{y}^i\right)^T\nabla F^i\left(\bar{w}^i\right) \bar{x}.\end{equation}
\end{itemize}
In addition, if  $(\bar{x},(\bar{y},\bar{w}))$ is optimal, then  $(\bar{x},(\bar{y},(\bar{x})_{i\in I}))$ is also optimal.
\end{restatable}
\pointtoproof{optimalityconditions:proof}

Three observations are important for Theorem~\ref{optimalityconditions}. The first one is that Theorem~\ref{optimalityconditions} is not a strong duality theorem because it gives optimality conditions for a primal-dual pair and not for a primal solution. Ensuring a primal optimal solution can be completed to a primal-dual pair does require additional qualification conditions, which is beyond the scope of this paper\footnote{Here we focus on recovering  solutions  returned by the solver and not on conditions that enable the solver to return such  solutions.}. 

The second observation concerns values $\bar{w}^i$. The last statement on Theorem~\ref{optimalityconditions} says we may take $\bar{w}^i=\bar{x}$ for all $i\in I$ for any optimal solution. Hence, we can safely replace all ${w}^i$ variables in dual problem \eqref{dualproblem} with a single set of variables $w\in \mathbb{R}^J$. As illustrated in Sections~\ref{quadraticconstsec} and \ref{lagrangiansection}, adding this restriction of a single set  of dual variables $w$ for all constraints results in a dual form closer to Lagrangian duality. We keep the module-wise copies for two reasons. First, it adds additional flexibility to the dual problem (e.g. if we happen to have a good dual bound from a solution with non-equal $\bar{w}^i$ there is no need to find the common $\bar{w}$). Second it keeps the module-wise  separation of dual variables common to conic and Fenchel duality where variables for each $i\in I$ are only linked by a stationarity equation (e.g. see Sections~\ref{conic_constraints} and 
\ref{connectiontofenchelsection}).

The third observation is that condition \eqref{optimaliticonditioneq} always imply feasibility of the primal solution $\bar{x}$ and the module-wise dual solution $(\bar{y}^i,\bar{w}^i)$. When these feasibilities depend only on indicator functions we can remove them from \eqref{optimaliticonditioneq} to obtain the following complementary slackness conditions.

\begin{definition}[Complementary Slackness]
We say $\bar{x}\in \mathbb{R}^J$ and $(\bar{y},\bar{w})\in \mathbb{R}^{M}\times \mathbb{R}^{I\times J}$ (feasible or infeasible) satisfy \emph{complementary slackness} for $i\in I$ if and only if
\begin{equation}
g_i^0\left(F^i\left(\bar{x}\right)\right)+\mathcal{O}_{F^i,g_i}^0(-\bar{y}^i, \bar{w}^i)=-\left(\bar{y}^i\right)^T\nabla F^i\left(\bar{w}^i\right) \bar{x}.\label{cscondition}
\end{equation}
\end{definition}

As illustrated in Sections~\ref{secificssection} and \ref{lagrangiansection}, \eqref{cscondition} reduces to familiar conditions for specific classes of functions. Furthermore, we recover the following familiar  optimality conditions in general.
\begin{restatable}{corollary}{csoptimalitycorollary}\label{csoptimalitycorollary}
A pair $(\bar{x},(\bar{y},\bar{w}))\in \mathbb{R}^J\times\left( \mathbb{R}^{M}\times \mathbb{R}^{I\times J}\right)$ is optimal if and only if
\begin{enumerate}
    \item $\bar{x}$ is primal feasible, 
   \item $(\bar{y}^i,\bar{w}^i)$ is module-wise dual feasible for each $i\in I$, 
   \item $(\bar{x},(\bar{y}^i,\bar{w}^i))$ satisfy complementary slackness condition \eqref{cscondition} for all $i\in I$,
   \item $(\bar{y},\bar{w})$ satisfies stationarity condition \eqref{stationarity}.
\end{enumerate}
\end{restatable}
\pointtoproof{csoptimalitycorollary:proof}

Finally, complementary slackness condition \eqref{cscondition} and stationarity condition \eqref{stationarity} imply equality of the indicator-ignoring objective values \eqref{primalobjmin} and \eqref{dualobjectivemin} even for infeasible solutions. In particular, through this result we can recover the familiar property of equal objectives for complementary solutions in LP as described in Section~\ref{LP:properties:section}. 

\begin{restatable}{lemma}{equalobjlemma}\label{equalobjlemma}
Let $\bar{x}\in \mathbb{R}^J$ and  $(\bar{y},\bar{w})\in \mathbb{R}^{M}\times \mathbb{R}^{I\times J}$ (feasible or infeasible) be such that $(\bar{y},\bar{w})$ satisfies   \eqref{stationarity} and $(\bar{x},(\bar{y},\bar{w}))$ satisfy  \eqref{cscondition} for all $i\in I$. Then, $\operatorname{obj_P}(\bar{x})=\operatorname{obj_D}(\bar{y},\bar{w})$.
\end{restatable}
\pointtoproof{equalobjlemma:proof}

\section{Rays and infeasibility certificates}\label{rayssection}
We use the following definition to  succinctly describe rays and infeasibility certificates. 
\begin{definition}
For any finite set $L$ and proper closed convex function $h:\mathbb{R}^{L} \to {\mathbb{R}\cup\{+\infty\}}$ the  \emph{asymptotic} or \emph{horizon} function of $h$ is $h_\infty:\mathbb{R}^L\to {\mathbb{R}\cup\{+\infty\}}$  given by 
    \[h_\infty\left(\bar{d}\right):=\sup_{t>0} \frac{h\left(\bar{z}+t \bar{d}\right)-h\left(\bar{z}\right)}{t},\]
where $\bar{z}\in \mathbb{R}^L$ is an arbitrary point such that $h(\bar{z})<+\infty$.
\end{definition}

The notions of primal/dual rays for \eqref{primalproblem}/\eqref{dualproblem} can be defined as follows.
\begin{definition}
A \emph{primal ray} for \eqref{primalproblem} is $\hat{x}\in \mathbb{R}^J$ such that \begin{equation}\label{primalraydef}c\cdot \hat{x}+\sum_{i\in I} (g_i\circ F^i)_\infty(\hat{x})<0.\end{equation}

A \emph{dual ray} for \eqref{dualproblem} is $(\hat{y},\hat{w})\in \mathbb{R}^{M}\times \mathbb{R}^{I\times J}$ such that
\begin{subequations}\label{dualraydef}
\begin{align}
\sum_{i\in I} \nabla F^i(\hat{w}^i)^T \hat{y}^i&= 0\label{dualrayeq}\\ 
    -\sum_{i \in I} F^i_\nabla(-\hat{y}^i, \hat{w}^i) +(g_i^*)_\infty\left(-\hat{y}^i\right)&> 0.\label{dualrayineq}
\end{align}
\end{subequations}
\end{definition}

As expected, together with the existence of a feasible solution, rays imply unboundedness of the associated problem (i.e. the first property of rays described for LP in Section~\ref{LP:properties:section}).

\begin{restatable}[Unboundedness Conditions]{theorem}{unboundedtheo}\label{unboundedtheo}
If the primal problem has a feasible solution $\bar{x}\in\mathbb{R}^J$ and a primal ray $\hat{x}\in\mathbb{R}^J$, then the primal problem is unbounded.

If the dual problem has a feasible solution $(\bar{y},\bar{w})\in \mathbb{R}^{M}\times \mathbb{R}^{I\times J}$ and a dual ray $(\hat{y},\hat{w})\in \mathbb{R}^{M}\times \mathbb{R}^{I\times J}$ such that
\begin{equation}\label{rayequalnabla}
\nabla F^i\left(\bar{w}^i\right)^T \hat{y}^i=\nabla F^i\left(\hat{w}^i\right)^T \hat{y}^i \quad \forall i \in I\
\end{equation}
then the dual problem is unbounded.
\end{restatable}
\pointtoproof{unboundedtheo:proof}

Because Theorem~\ref{weakduality}, existence of a primal feasible solution and ray implies infeasibility of the dual problem and vice-versa. However, existence of ray for one problem implies infeasibility of the other even in the absence of the feasible solution (i.e. the second role of LP rays  from Section~\ref{LP:properties:section}). 

\begin{restatable}[Infeasibility Certificates]{theorem}{infeasibilitytheo}\label{infeasibilitytheo}
If a primal ray exists, then the dual problem is infeasible, and if a dual ray exists, then the primal problem is infeasible.
\end{restatable}
\pointtoproof{infeasibilitytheo:proof}

\section{Specific constraints}\label{secificssection}

 \subsection{Conic constraints (including one-sided linear inequalities)}\label{conic_constraints}
 Currently MathOpt only supports conic constraints for the non-negative orthant (i.e. one-sided linear inequalities) and   the second order cone (SOC) or conic quadratic cone \citep{ben2001lectures,aps2020mosek}. However, our results are applicable for any closed convex cone.
 
\begin{restatable}{proposition}{conicspecific}\label{conicspecific}
 Let $K_i\subseteq \mathbb{R}^{M_i}$ be a closed convex cone, 
 $K^\circ_i:=\{y\in \mathbb{R}^{M_i}\,:\,y\cdot x\leq 0\quad\forall x\in K_i\}$ be the polar cone of $K_i$, $F^i:\mathbb{R}^J\to\mathbb{R}^{M_i}$ be such that $F^i(x)=A^ix-b^i$ for $A^i\in \mathbb{R}^{M_i\times J}$ and $b^i\in \mathbb{R}^{M_i}$ and $g_i:\mathbb{R}^{M_i}\to\mathbb{R}\cup\{+\infty\}$ be such that  $g_i=\delta_{K_i}$. If $g_i\circ F^i$ satisfies \eqref{assumptioneq:gofproper}, then $g_i^0\equiv 0$, $g_i^*=(g_i^*)^\infty=\delta_{K^\circ_i}$, $(g_i^*)^0\equiv 0$, $\nabla F^i\left(w^i\right)=A^i$,  $  F^i_\nabla\left(y^i, w^i\right)=b^i\cdot y^i$, and $\left(g_i\circ F^i\right)_\infty(\hat{x})=\delta_{K_i}\left(A^i\hat{x}\right)$.
 \end{restatable}
\pointtoproof{conicspecific:proof}

To illustrate the use of Proposition~\ref{conicspecific}, note that primal problem \eqref{primalproblem} with all $F^i$ and $g_i$ as in Proposition~\ref{conicspecific} reduces to 
\begin{align*}
\operatorname{opt}_P:&=
\begin{alignedat}[t]{3}
&\inf_{x\in \mathbb{R}^J}& c\cdot x&\\
&s.t.&&\\
&& A^ix-b^i&\in K_i \quad \forall i\in I
\end{alignedat}
\end{align*}
and the associated dual problem  \eqref{dualproblem} reduces to
\begin{subequations}
\begin{alignat}{3}
\operatorname{opt}_D:=
&\sup_{y\in \mathbb{R}^M}&  \sum_{i \in I} b^i\cdot y^i  & \\\
&s.t.&&\notag\\
&& \sum_{i\in I} (A^i)^{T}y^i&= c,\quad -y^i&\in K_i^\circ\quad \forall i \in I.
\end{alignat}
\label{conicdual}
    \end{subequations}
If we let  $K^*_i:=-K^\circ_i=\{y\in \mathbb{R}^{M_i}\,:\,y\cdot x\geq 0\quad\forall x\in K_i\}$ be the dual cone of $K_i$,
we recognize in \eqref{conicdual} the standard conic dual (i.e. \eqref{conicdual} with $-y^i\in K_i^\circ$ replaced by $y^i\in K_i^*$).

As expected, because $F^i$ is an affine function, variables $w^i$ do not appear in \eqref{conicdual}. 

Primal feasibility becomes the standard conic primal feasibility $A^ix-b^i\in K_i$ for all $i\in I$, and dual feasibility becomes the standard conic dual feasibility $\sum_{i\in I} (A^i)^{T}y^i= c$ and $y^i\in K_i^*$ for all $i\in I$. Complementary slackness is the standard conic complementary slackness 
$\bar{y}^i\cdot (A^i \bar{x} -b^i)=0$  for all  $i \in I$. In addition, primal ray definition \eqref{primalraydef} becomes the standard conic primal ray condition
\begin{equation*}A^i\hat{x}\in K_i\quad\forall i\in I \quad\text{and}\quad c\cdot \hat{x}<0,\end{equation*}
and dual ray definition \eqref{dualraydef} becomes the standard conic dual ray condition
\begin{align*}
    \sum_{i\in I} (A^i)^{T}\hat{y}^i &=0,\quad \hat{y}^i\in K_i^*\quad\forall i \in I \\ \sum_{i \in I} b^i\cdot \hat{y}^i &>0.
\end{align*}
Again, we can see that variables $w^i$ are also not needed to describe dual rays.

 \subsection{Two-sided (or ranged) linear inequalities}\label{twosideddec}

 \begin{restatable}{proposition}{twosidedspecific}\label{twosidedspecific}
 Let $a^i\in \mathbb{R}^J$, $l_i\in \mathbb{R}$ and $u_i\in \mathbb{R}$ be such that $l_i\leq u_i$,
 $F^i:\mathbb{R}^J\to\mathbb{R}$ be such that $F^i(x)=a^i\cdot x$ and $g_i:\mathbb{R}\to\mathbb{R}\cup\{+\infty\}$ be such that $g_i=\delta_{[l_i, u_i]}$. Then $g_i^0\equiv 0$, $g_i^*(y_i)=(g_i^*)^\infty(y_i)=(g_i^*)^0(y_i)=\max\{l_i y_i,u_i y_i\}$, $\nabla F^i\left(w^i\right)=a^i$,  $  F^i_\nabla\equiv 0$, and $\left(g_i\circ F^i\right)_\infty\left(\hat{x}\right)=\delta_{\{0\}}\left(a^i\cdot \hat{x}\right)$.
 \end{restatable}
\pointtoproof{twosidedspecific:proof}

To illustrate the use of Proposition~\ref{twosidedspecific}, note that primal problem \eqref{primalproblem} with all $F^i$ and $g_i$ as in Proposition~\ref{twosidedspecific} reduces to 
\begin{subequations}   
\begin{alignat}{3}
\operatorname{opt}_P:=
&\inf_{x\in \mathbb{R}^J}& c\cdot x&\\
&s.t.&&\notag\\
&& l_i\leq a^i\cdot x&\leq u_i \quad \forall i \in I  
\label{twosidedprimal}
\end{alignat}
\end{subequations}
and the associated dual problem  \eqref{dualproblem} reduces to
\begin{subequations}
\label{twosideddual}
\begin{alignat}{3}
\operatorname{opt}_D:=
&\sup_{y\in \mathbb{R}^M}& \quad \sum_{i \in I} \min\{l_i y_i,&u_i y_i\}    \\\
&s.t.&&\notag\\
&& \sum_{i\in I} a^iy_i&= c.
\end{alignat}
\end{subequations}
As expected, because $F^i$ is a linear function, variables $w^i$ do not appear in \eqref{twosideddual}. We can also see that the  dual \eqref{standard:ub:dual:lp} for LPs with lower and upper variable bounds is a special case of \eqref{dualproblem}. In particular, \eqref{twosideddual} extends  \eqref{standard:ub:dual:lp} when the equality constraints of \eqref{standard:ub:dual:lp} are absent.

Primal feasibility becomes  $l_i\leq a^i\cdot x\leq u_i$ for all $i\in I$, all module-wise dual solutions $\bar{y}_i$ are feasible, and feasibility of a full dual solution $\bar{y}$ reduces to satisfying stationarity condition $\sum_{i\in I} a^iy_i= c$. Complementary slackness becomes 
\begin{equation}\label{signed:cs}
    \min\{l_i\bar{y}_i,u_i\bar{y}_i\}=\bar{y}_i a^i\cdot \bar{x} \quad \forall i \in I.
\end{equation}

Finally, primal ray definition \eqref{primalraydef} becomes
\begin{equation}a^i\cdot\hat{x}=0\quad\forall i\in I \quad\text{and}\quad c\cdot \hat{x}<0,\end{equation}
and dual ray definition \eqref{dualraydef} becomes
\begin{equation}
     \sum_{i\in I} a^i\hat{y}_i =0 \quad\text{ and }\quad  \sum_{i \in I} \min\{l_i \hat{y}_i,u_i \hat{y}_i\}   >0.
\end{equation}
Again, we can see that variables $w^i$ are also not needed to describe dual rays.

\subsection{Quadratic Objective}\label{quadraticobjsec}

In Table~\ref{initialfunctions} we described (homogeneous) quadratic functions in matrix notation. However, such matrix notation is not completely standardized (e.g.  the associated matrix corresponds to the Hessian of the quadratic function only if the function is scaled by $1/2$). Similarly, we also want contracts that are applicable when a non-matrix notational convention is used. Fortunately these issues can be resolved by using the following straightforward properties of quadratic functions.

\begin{definition}\label{quadraticform:def}
   We say $h:\mathbb{R}^J\to\mathbb{R}$ is a \emph{convex quadratic form} if it is a convex homogeneous polynomial of degree two (i.e. a convex polynomial whose non-zero terms all have degree $2$).
\end{definition}

\begin{lemma}\label{quadraticform:lemma}
    Let $h:\mathbb{R}^J\to\mathbb{R}$ be a convex quadratic form. Then, for any $w,w'\in \mathbb{R}^J$ we have $h(w)=\frac{1}{2}w^T \nabla^2h(w')w$, $\nabla h(w)=\nabla^2h(w')w$ and $\nabla h\left(w\right)\cdot w=2h_i\left(w\right)$. 
    
    In addition, $\nabla h\left(w\right)$ has a simple description for matrix- or non-matrix-based descriptions of $h$:
\begin{itemize}
    \item If $h(w)=w^TQw$ for $Q\in \mathbb{R}^{J\times J}$, then $\nabla h(w)=2 Qw$.
    \item If $h(w)=(1/2)w^TQw$ for $Q\in \mathbb{R}^{J\times J}$, then $\nabla h(w)=Qw$.
    \item If $h(w)=\sum_{(j,j')\in K} q_{j,j'}w_jw_{j'}$ for $K:=\{(i,j')\in J\times J\,:\, i\leq j'\}$ and $q_{j,j'}\in \mathbb{R}$  for all $(j,j')\in K$, then  $(\nabla h(w))_j=q_{j,j'}+\sum_{j'\in J} q_{\min\{j,j'\},\max\{j,j'\}} w_{j'}$ for all $j\in J$.
\end{itemize}
\end{lemma}

Thanks to  \cref{quadraticform:lemma} we can describe our contracts for quadratic objectives (and constraints) abstractly using $h$ and $\nabla h$, and later concretize them for a particular notational convention.

 \begin{restatable}{proposition}{simplequadraticobjlemma}\label{simplequadraticobjlemma}
Let $h_i:\mathbb{R}^J\to\mathbb{R}$ be a convex quadratic form, $F^i=h_i$  and $g_i:\mathbb{R}\to\mathbb{R}\cup\{+\infty\}$ be such that $g_i(z)=z$.  Then $g_i^0(z)=g_i(z)$, $g_i^*=\delta_{\{1\}}$,  $(g_i^*)^0\equiv 0$, $(g_i^*)^\infty=\delta_{\{0\}}$, $\nabla F^i=\nabla h_i$,  $  F^i_\nabla\left(y_i, w^i\right)=y_ih_i(w^i)$, and $\left(g_i\circ F^i\right)_\infty\left(\hat{x}\right)=\delta_{\{0\}}\left(\nabla h_i\left(\hat{x}\right)\right)$.
\end{restatable}
\pointtoproof{simplequadraticobjlemma:proof}

To illustrate the use of Proposition~\ref{simplequadraticobjlemma}, note that primal problem \eqref{primalproblem} for $I=\{i_0\}$ and $F^{i_0}$ and $g_{i_0}$ as in Proposition~\ref{simplequadraticobjlemma} reduces to 
\begin{equation}
\operatorname{opt}_P:=\inf_{x\in \mathbb{R}^J}  c\cdot x + h_{{i}_0}(x)=\inf_{x\in \mathbb{R}^J}  c\cdot x + g_{i_0}\left(F^{i_0}(x)\right)
\end{equation}
and the associated dual problem  \eqref{dualproblem} reduces to
\begin{subequations}\label{quadobjdual}
\begin{alignat}{3}
\operatorname{opt}_D:
=&\sup_{\substack{y_{i_0}\in \mathbb{R},\\ w^{i_0}\in \mathbb{R}^J}}\quad&   y_{i_0}h_{i_0}\left(w^{i_0}\right)\\
&s.t.&&\notag\\
&&y_{i_0} \nabla h\left(w^{i_0}\right)&= c,\quad -y_{i_0}=1.
\label{quadobjdualfinal}
\end{alignat}
\end{subequations}
Because $F^i$ is not an affine function, variables $w^i$ do appear in \eqref{quadobjdual}. 

All primal solutions are feasible (as the primal problem is unconstrained linear plus quadratic minimization), dual feasibility is $-\bar{y}_{i_0}=1$, and complementary slackness is 
\begin{equation}\label{quadobjcs}
    h_{i_0}(\bar{x})-
    \bar{y}_{i_0}h_{i_0}\left(\bar{w}^{i_0}\right)=-\bar{y}_{i_0} \nabla h_{i_0}\left(\bar{w}^{i_0}\right)\cdot\bar{x}.
\end{equation}
In addition, primal ray definition \eqref{primalraydef} becomes
$\nabla h_{i_0}\left(\hat{x}\right)=0$  and $c\cdot \hat{x}<0$. Similarly,
 dual ray definition \eqref{dualraydef} becomes
\begin{equation*}
     \hat{y}_{i_0}\nabla h\left(\hat{w}^{i_0}\right) =0,\quad  \hat{y}_{i_0} =0,\quad  \hat{y}_{i_0}h_{i_0}\left(\hat{w}^{i_0}\right)   >0.
\end{equation*}
In particular, we can check that a dual ray does not exist as expected from the problem being unconstrained.

As stated by Lemma~\ref{equalobjlemma}, we can check that any pair $(\bar{x}, (\bar{y},\bar{w}))$ that satisfies \eqref{quadobjcs} and stationarity condition $\bar{y}_{i_0} \nabla h_{i_0}(\bar{w}^{i_0})= c$, has $\operatorname{obj}_P(\bar{x})=c\cdot \bar{x}+h(\bar{x})=\bar{y}_{i_0}h_{i_0}\left(\bar{w}^{i_0}\right)=\operatorname{obj}_D(\bar{y},\bar{w})$,
even if $\bar{y}$ does not satisfy dual feasibility constraint $-\bar{y}_{i_0}=1$. However, if $y$ is allowed to be an (a priori)  unfixed variable, then \eqref{quadobjdualfinal} includes non-convex functions in its objective and constraints. 

To get a \emph{clearly} convex re-formulation of the dual problem we can eliminate variable $y_{i_0}$ using the dual feasibility equation $-y_{i_0}=1$ to obtain the more familiar dual problem (e.g. see \cite[(3)]{vandenberghe2010cvxopt} and \cite[Section 6.3]{roos2020universal}) given by
\begin{subequations}
    \begin{alignat}{3}
    \label{quadobjdualconvex}
\operatorname{opt}_D:=&\sup_{ w^{i_0}\in \mathbb{R}^J}\quad&   &-h_{i_0}\left(w^{i_0}\right)\\
&s.t.&&\notag\\
&&0 &= c+\nabla h_{i_0}\left(w^{i_0}\right)
\end{alignat}
\end{subequations}
This reformulation does not need variable $y_{i_0}$ and all its dual solutions are module-wise feasible (i.e. dual feasibility reduces to the stationarity equation). In addition, complementary slackness conditions \eqref{quadobjcs} become
$h_{i_0}(\bar{x})+h_{i_0}\left(\bar{w}^{i_0}\right)= \nabla h_{i_0}\left(\bar{w}^{i_0}\right)\cdot\bar{x}$,
which, by Lemma~\ref{quadraticform:lemma}, is equivalent to
\begin{equation}\label{quadobjcsfinal}
   \left( \nabla h_{i_0}\left(\bar{x}\right)-\nabla h_{i_0}\left(\bar{w}^{i_0}\right)\right)\cdot\left(\bar{x}-\bar{w}^{i_0}\right)=0.
\end{equation}
In particular, complementary slackness condition \eqref{quadobjcsfinal} holds if $\bar{x}=\bar{w}^{i_0}$ in which case the stationarity conditions become the standard optimality conditions for linear plus quadratic unconstrained minimization given by $c=-\nabla h_{i_0}\left(\bar{x}\right)$. However, the complementary slackness connection between $\bar{x}$ and $\bar{w}^{i_0}$ can also be achieved through $\nabla h_{i_0}\left(\bar{x}\right)=\nabla h_{i_0}\left(\bar{w}^{i_0}\right)$, which may be weaker than $\bar{x}=\bar{w}^{i_0}$ when $\nabla^2 h$ is singular.

 \subsection{Quadratic Constraint}\label{quadraticconstsec}

 \begin{restatable}{proposition}{quadraticconstrlemma}\label{quadraticconstrlemma}
Let $h_i:\mathbb{R}^J\to \mathbb{R}$ be a \emph{convex quadratic form}, $q^i\in \mathbb{R}^{J}$, $u_i\in \mathbb{R}$, $F^i(x)={h_i(x)+q^i\cdot x -u_i}$ and
 $g_i=\delta_{\mathbb{R}_-}$.   Then,  $g_i^0\equiv 0$,  $g_i^*=(g_i^*)^\infty=\delta_{\mathbb{R}_+}$,  $(g_i^*)^0\equiv 0$,  $\nabla F^i(w^i)=\nabla h_i(w^i)+q^i$,  $  F^i_\nabla\left(y^i, w^i\right)=\left(h_i\left(w^i\right)+u\right)y_i$, and $\left(g_i\circ F^i\right)_\infty=\delta_\mathcal{D}$ for $\mathcal{D}:=\left\{d\in \mathbb{R}^J\,:\, \nabla h_i(d)=0,\quad q^i\cdot d \leq 0\right\}$.
\end{restatable}
\pointtoproof{quadraticconstrlemma:proof}

To illustrate the use of Proposition~\ref{quadraticconstrlemma}, note that primal problem \eqref{primalproblem} for $I=\{i_0\}$ and $F^{i_0}$ and $g_{i_0}$ as in Proposition~\ref{quadraticconstrlemma} reduces to 
\begin{align*}
\operatorname{opt}_P:&=\begin{alignedat}[t]{3}&\inf_{x\in \mathbb{R}^J}\quad&   c\cdot x\\
&s.t.&&\\
&&h_{i_0}(x)+q^{i_0}\cdot x &\leq u_{i_0}&
\end{alignedat}
\end{align*}
and the associated dual problem  \eqref{dualproblem} reduces to
\begin{subequations}
\label{quadconstrdualtoref}
\begin{align}
\operatorname{opt}_D:
=&\sup_{{y_{i_0}\in \mathbb{R}, w^{i_0}\in \mathbb{R}^J}}    \left(h_{i_0}\left(w^{i_0}\right)+u_{i_0}\right)y_{i_0}\\
&s.t.\notag\\
&\quad\quad y_{i_0} \left(\nabla h_{i_0}\left(w^{i_0}\right)+q^{i_0}\right)= c,\quad y_{i_0}\leq 0.
\end{align}
\end{subequations}
Because $F^i$ is not an affine function, variables $w^i$ do appear in \eqref{quadconstrdualtoref}. 

Primal feasibility becomes  $h_{i_0}(\bar{x})+q^{i_0}\cdot \bar{x} \leq u_{i_0}$, dual feasibility is $\bar{y}_{i_0}\leq 0$ plus $y_{i_0} \left(\nabla h_{i_0}\left(w^{i_0}\right)+q^{i_0}\right)= c$  and using \cref{quadraticform:lemma} we can write complementary slackness as 
\begin{align}
\notag
\bar{y}_{i_0} \left(h\left(\bar{w}^{i_0}\right)\right.&+\left.q^{i_0}\cdot \bar{w}^{i_0} -u_{i_0}\right)\\&=\bar{y}_{i_0}\left(\nabla h\left(\bar{w}^{i_0}\right)+q^{i_0}\right)\cdot\left(\bar{w}^{i_0}-\bar{x}\right).\label{quadraticconstcs}
\end{align}
In addition, primal ray definition \eqref{primalraydef} becomes
\begin{equation}
\nabla h_{i_0}\left(\hat{x}\right)=0, \quad q^{i_0}\cdot \hat{x}\leq 0 \quad\text{and}\quad c\cdot \hat{x}<0,\end{equation}
and dual ray definition \eqref{dualraydef} becomes
$\hat{y}_{i_0} \left(\nabla h_{i_0}\left(\hat{y}^{i_0}\right)+q^{i_0}\right)=0$,     $\hat{y}_{i_0} \leq 0$ and $\left(h_{i_0}\left(\hat{y}^{i_0}\right)+u_{i_0}\right)\hat{y}_{i_0}  >0$.

Note that complementary slackness condition \eqref{quadraticconstcs} becomes the standard Lagrangian duality complementary slackness condition when $\bar{w}^{i_0}=\bar{x}$. However, dual feasibility $\bar{y}_{i_0}\leq 0$ is the negation of standard non-negative Lagrangian multipliers. This is a side effect of our alignment with sign conventions for standard conic and linear programming duality (see Section~\ref{conic_constraints} and \ref{conicsectionmax}) and the fact that these conventions do not perfectly align with Lagrangian multiplier conventions (e.g. see \cite[Section 5.2.1]{boyd2004convex}).

Now, similarly to the quadratic objective from Section~\ref{quadraticobjsec}, the dual problem \eqref{quadconstrdualtoref} appears non-convex and hence needs a convex reformulation. As noted in \cite{roos2020universal}, this can be done through the change of variable $z^{i_0}=y_{i_0}w^{i_0}$ and the following function associated to $h_{i_0}$.

 \begin{restatable}{proposition}{quadraticformulationlemma}\label{quadraticformulationlemma}
Let $h_i:\mathbb{R}^J\to \mathbb{R}$ be a \emph{convex quadratic form}, and $\mathcal{G}_{h^i}:\mathbb{R}^J\times \mathbb{R}\to \mathbb{R}\cup\{+\infty\}$ be the closure of the perspective function of $h_i$ given by
\[\mathcal{G}_{h^i}\left(z^i, \mu_i\right):=\begin{cases}\frac{h_i\left(z^i\right)}{\mu_i}& \mu_i>0\\
0 & z^i=0, \quad \mu_i=0\\
\infty & \text{o.w.}\end{cases}.\]
Then $\mathcal{G}_{h^i}$ is a proper closed convex function.
In addition, for
\begin{equation}
\label{quadraticformulationlemmaconversion}
z^i\left(w^i,y_i\right):=y_i w^i,\quad
w^i\left(z^i,y_i\right):=\begin{cases}z^i/y_i &y_i<0 \\0&\text{o.w.}\end{cases}
    \end{equation}
we have 
$-\mathcal{G}_{h^i}\left(z^i\left(w^i,y_i\right), -y_i\right)=y_i h_i\left(w^i\right)$ and $
-\mathcal{G}_{h^i}\left(z^i, -y_i\right)=y_i h_i\left(w^i\left(z^i,y_i\right)\right)$ for any $y_i
\leq 0$.
\end{restatable}
\pointtoproof{quadraticformulationlemma:proof}

Using \cref{quadraticformulationlemma} and the fact that $y_i\nabla h_i\left(w^i\right)=\nabla h_i\left(y_i w^i\right)$ for any quadratic form $h_i$ we can reformulate \eqref{quadconstrdualtoref}  as the convex optimization problem 
\begin{subequations}\label{quadconstrdualconv}
\begin{alignat}{3}
\operatorname{opt}_D:=&\sup_{\substack{y_{i_0}\in \mathbb{R},\\ z^{i_0}\in \mathbb{R}^J}}    -\mathcal{G}_{h^{i_0}}\left(z^{i_0}, -y_{i_0}\right)+u_{i_0}y_{i_0}\\
&s.t.\notag\\
&\quad\quad  \nabla h_{i_0}\left(z^{i_0}\right)+q^{i_0}y_{i_0}= c,\quad y_{i_0}\leq 0,
\end{alignat}
\end{subequations}
and we can convert between the solution for \eqref{quadconstrdualtoref} and \eqref{quadconstrdualconv} through \eqref{quadraticformulationlemmaconversion}.

\section{Connection to other duality frameworks}\label{connectiontodualssection}

\subsection{Connection to Fenchel duality}\label{connectiontofenchelsection}
The Fenchel dual of \eqref{primalproblem} can be written as
\begin{subequations}\label{fencheldual}
\begin{alignat}{3}
\operatorname{opt}_{D_{\text{Fenchel}}}:=&\sup_{v\in \mathbb{R}^{I\times J}}\quad&  -\sum_{i \in I}  \left(g_i\circ F^i\right)^*\left(-v^i\right) &\\\
&s.t.\notag&&\\
&& \sum_{i\in I} v^i= c. \label{fencheldualstationarity}
\end{alignat}
\end{subequations}

As noted in \cite{roos2020universal}, getting a practical version of \eqref{fencheldual} reduces to getting practical formulations of the conjugate functions $\left(g_i\circ F^i\right)^*$. In \cite{roos2020universal} it is shown how such practical formulations can be obtained in an ad-hoc manner for various classes of functions. In contrast, through functions $\mathcal{O}_{F^i,g_i}$ (see \cref{dual:obj:def}), \eqref{dualproblem} provides a more mechanical construction (we only need a formulation of $F^i$, $\nabla F^i$ and $g_i^*$ instead of $\left(g_i\circ F^i\right)^*$, which as argued below should be easier). In this section we explore how $\mathcal{O}_{F^i,g_i}$ appears from a generic approximation of $\left(g_i\circ F^i\right)^*$ and discuss when this approximation is exact. To achieve this we need the following definitions of constraint qualification from \cite{rockafellar2009variational}.

\begin{definition}
$h:\mathbb{R}^J\to \mathbb{R}\cup\{+\infty\}$ is \emph{piecewise linear-quadratic} (PWLQ) if and only if
\[
h(x)=\begin{cases} h_l(x)+q^l\cdot x -u_l&x\in P_l, l\in L\\+\infty&\text{o.w.}\end{cases},
\]
where $L$ is finite and for each $l\in L$, $P_l$ is a polyhedron, $h_l$ is a quadratic form, $q^l\in\mathbb{R}^J$ and $u_l\in \mathbb{R}$. 

For $F^i:\mathbb{R}^J\to \mathbb{R}^{M_i}$ and $g_i:\mathbb{R}^{M_i} \to {\mathbb{R}\cup\{+\infty\}}$ let
$U_{i}:=\left\{u\in \mathbb{R}^{M_i}\,:\, \inf_{x\in \mathbb{R}^J}g_i\left(F^i(x)-u\right)<+\infty\right\}$.
We say $g_i\circ F^i$ satisfies constraint qualification if and only if
\begin{subequations}
\label{constraintqual}
\begin{align}
        0&\in \operatorname{int}\left(U_{i}\right), \text{or}\label{constraintqual:a}\\
        0&\in U_{i}, \text{ and $g_i\circ F^i$ is PWLQ}.\label{constraintqual:b}
        \end{align}
\end{subequations}
\end{definition}
Note that for the functions in Section~\ref{secificssection}/Table~\ref{initialfunctions},  \eqref{constraintqual:a} reduces to standard module-wise Slater conditions and \eqref{constraintqual:b} includes module-wise polyhedral Slater  conditions.

The following corollary establishes the connection between $\mathcal{O}_{F^i,g_i}$ and $\left(g_i\circ F^i\right)^*$. 
\begin{restatable}{corollary}{connectiontofenchelcoro}\label{connectiontofenchelcoro}
For $F^i:\mathbb{R}^J\to \mathbb{R}^{M_i}$ and $g_i:\mathbb{R}^{M_i} \to {\mathbb{R}\cup\{+\infty\}}$ that satisfy \eqref{assumptioneq}, and for all $\bar{v}^i\in \mathbb{R}^J$ 
\begin{equation}\label{connectiontofenchelcoroineq}
   - \left(g_i \circ F^i\right)^*\left(-\bar{v}^i\right)\geq \hspace{-3pt} \begin{alignedat}[t]{3}&\hspace{-2pt}\sup_{\substack{y^i \in \mathbb{R}^{M_i},\\w^i\in \mathbb{R}^J}
   }-\mathcal{O}_{F^i,g_i}\left(-y^i, w^i\right)\\
   &\quad s.t.\\
   &\quad\quad \nabla F^i\left(w^i\right)^T y^i=\bar{v}^i.\end{alignedat}
\end{equation}

 In addition, for all $\bar{v}^i\in \mathbb{R}^J$ and $\bar{y}^i\in \operatorname{dom}(g_i^*)$,
  \begin{equation}\label{connectiontofenchelcoroineqopt}
     \left(\bar{y}^i\cdot F^i\right)^*\left(\bar{v}^i\right)=\sup_{w^i\in \mathbb{R}^J} \bar{v}^i\cdot w^i-\bar{y}^i\cdot F^i\left(w^i\right)
 \end{equation}
 is a convex optimization problem. 
 
 Finally, let $S\left(\bar{v}^i, \bar{y}^i\right)=\arg \max_{w^i\in \mathbb{R}^J}\bar{v}^i\cdot w^i-\bar{y}^i\cdot F^i\left(w^i\right)$ and consider the  condition 
 \begin{equation}\label{connectiontofenchelcoroineqopt:cond}
     \left(\bar{y}^i\cdot F^i\right)^*\left(\bar{v}^i\right)<+\infty\quad \Rightarrow \quad    S\left(\bar{v}^i, \bar{y}^i\right)\neq \emptyset,
 \end{equation}
which states that \eqref{connectiontofenchelcoroineqopt} attains its optimal value whenever it is finite. 

If \eqref{connectiontofenchelcoroineqopt:cond} holds for all $\bar{v}^i\in \mathbb{R}^J$ and $\bar{y}^i\in \operatorname{dom}(g_i^*)$, and $g_i\circ F^i$ satisfies  \eqref{constraintqual}, then 
 \eqref{connectiontofenchelcoroineq} holds at equality with $\sup$ replaced by $\max$.
\end{restatable}
\pointtoproof{connectiontofenchelcoro:proof}

Inequality \eqref{connectiontofenchelcoroineq} shows dual problem \eqref{dualproblem} is a conservative approximation of Fenchel dual \eqref{fencheldual} and  \begin{equation}\label{fencheldualapprox}
\operatorname{opt}_{D_{\text{Fenchel}}}\geq \operatorname{opt}_{D},
\end{equation}
with equality under the additional conditions in \cref{connectiontofenchelcoro}.
The following lemma shows that for  all functions in Table~\ref{initialfunctions} the more technical condition \eqref{connectiontofenchelcoroineqopt:cond} always holds, so to guarantee equality in \eqref{fencheldualapprox} we only need Slater conditions  \eqref{constraintqual}.

\begin{restatable}{lemma}{attainmentlemma}\label{attainmentlemma}
If $(F^i, g^i)$ is in Table~\ref{initialfunctions}, then \eqref{connectiontofenchelcoroineqopt:cond} holds for all $\bar{v}^i\in \mathbb{R}^J$ and $\bar{y}^i\in \operatorname{dom}(g_i^*)$.
\end{restatable}
\pointtoproof{attainmentlemma:proof}

Now, consider a dual solution $\bar{v}\in \mathbb{R}^{I\times J}$ for \eqref{fencheldual} that satisfies stationarity condition \eqref{fencheldualstationarity}. Evaluating the objective value of $\bar{v}$ to get a dual bound requires solving the convex optimization problems 
\begin{equation}\label{fenchelsubopt}
     \left(g_i\circ F^i\right)^*\left(-\bar{v}^i\right) :=\sup_{w^i\in \mathbb{R}^J}-\bar{v}^i\cdot w^i-g\left(F^i\left(w^i\right)\right)  .
\end{equation}
Conversely, for a dual solution $(\bar{y},\bar{w})\in \mathbb{R}^{M}\times \mathbb{R}^{I
\times J}$ for  \eqref{dualproblem} (that satisfies stationarity condition \eqref{stationarity}) we just need to evaluate $\mathcal{O}_{F^i,g_i}\left(-\bar{y}^i, \bar{w}^i\right)=-\left(\nabla F^i\left(\bar{w}^i\right)\bar{w}^i -  F^i\left(\bar{w}^i\right)\right)\cdot y^i+g^*_i\left(-\bar{y}^i\right)$,
which only requires solving the simpler convex optimization problems 
\begin{equation}
    g^*_i\left(-\bar{y}^i\right):=\sup_{z^i\in \mathbb{R}^{M_i}}-\bar{y}^i\cdot z^i-g\left(z^i\right).
\end{equation}
In particular, if we are only considering a fixed class of functions $g_i$ for which $g_i^*$ has a closed form (such as those in Table~\ref{initialfunctions}), then evaluating the objective value of $(\bar{y},\bar{w})$ does not require solving any optimization problem. In contrast, even if we restrict to functions $g_i$ with closed form  $g_i^*$ and affine functions $F^i$, \eqref{fenchelsubopt} may still be a non-trivial optimization problem whose optimal value may be unattained (e.g. if $g_i$ is the indicator function of a closed convex cone \eqref{fenchelsubopt} is a conic optimization problem). Hence,  \eqref{dualproblem} may be preferable over \eqref{fencheldual}, particularly when \eqref{fencheldualapprox} holds at equality.

\subsection{Connection to Lagrangian and generalized Lagrangian duality}\label{lagrangiansection}

We can fit \eqref{primalproblem} in the generalized nonlinear programming (GNLP) framework of \cite{rockafellar2023augmented} by letting $f_0:\mathbb{R}^J\to\mathbb{R}$, $g:\mathbb{R}^M\to\mathbb{R}\cup\{+\infty\}$, and $F:\mathbb{R}^J\to \mathbb{R}^M$ in \cite[(1.1)]{rockafellar2023augmented} be given by $f_0(x)=c\cdot x$, $g\left(\left(z^i\right)_{i\in I}\right)=\sum_{i\in I} g_i\left(z^i\right)$, and $F(x)=\left(F^i(x)\right)_{i\in I}$. Then the Lagrangian function $l:\mathbb{R}^J\times \mathbb{R}^M:\to \mathbb{R}\cup\{-\infty\}$ in \cite[(1.1)]{rockafellar2023augmented} becomes 
$l(x,y)= c\cdot x+\sum_{i\in I} y^i\cdot F^i(x)-g_i^*(y^i)$, which reduces to the standard Lagrangian function when  $|M_i|=1$ and $g_i=
\delta_{\mathbb{R}_-}$ for all $i\in I$.

Optimality condition $0\in \partial_x l(\bar{x},\bar{y})=0$ for $\inf_{x\in \mathbb{R}^J}l(x,y)$ (cf. \cite[(1.18)]{rockafellar2023augmented}) becomes
\begin{equation}\label{wolfeeq}
    c+ \sum_{i\in I}\nabla F^i\left(x\right)^T y^i= 0,
\end{equation}
which we can use to write the version of Wolfe's dual based on $\sup_{y\in \mathbb{R}^M}\inf_{x\in \mathbb{R}^J}l(x,y)$ given by 
\begin{equation}\label{wolfedual}
  \operatorname{opt}_D:=\sup_{y\in \mathbb{R}^{M}, x\in \mathbb{R}^{ J}}\left\{l(x,y)\,:\, \eqref{wolfeeq}\right\}.
\end{equation}
Then, multiplying \eqref{wolfeeq} by $x$ to replace $c\cdot x$ in the definition of $l(x,y)$ we can see that $l(x,y)$ in \eqref{wolfedual} can be replaced by 
$l(x,y)=-\sum_{i\in I}F^i_\nabla\left(y^i, x\right)+g_i^*(y^i)$ and hence \eqref{wolfedual} is a version of dual \eqref{dualproblem} where $-y$ is replaced by $y$ and $w^i=x$ for all $i\in I$. That is, with the exception of having copies of $x$ for each $i\in I$, dual \eqref{dualproblem} is the analog of Wolfe's dual for the GNLP Lagrangian dual of  \cite{rockafellar2023augmented} (or more precisely for the version of Wolfe's dual by \cite{dorn1960duality} as described in \cite{wolfe1961duality}).

In particular, for the standard Lagrangian setting (i.e.  $|M_i|=1$ and $g_i=
\delta_{\mathbb{R}_-}$ for all $i\in I$), complementary slackness condition \eqref{cscondition} without the module-wise copies $w^i$ (i.e. with $w^i=x$ for all $i\in I$) simplifies to $y_i
 F^i(x)=0$, which we recognize as the standard Lagrangian complementary slackness condition.  

\subsection{Connection to two-sided optimality conditions}\label{knitro:cs:section}

\cite{knitrotermination} introduces optimality conditions for optimization problems with two-sided nonlinear inequalities. In the convex case, the \emph{true two-sided inequalities} (i.e. those for which both bounds are finite) can only be two-sided linear inequalities so we focus our comparison to the context of Section~\ref{twosideddec}. Through $
\lambda_i=-y_i$, the optimality conditions of \cite{knitrotermination} for  \eqref{twosidedprimal}
are
\begin{subequations}   
\begin{alignat}{3}
    l_i\leq a^i\cdot \bar{x}&\leq u_i &\quad&\forall i\in I\label{knitro:pf}\\
    \sum_{i\in I} a^i\bar{y}_i &=c\label{knitro:stat}\\
   \bar{y}_i \min\{a^i\cdot \bar{x}-l_i,u_i-a^i\cdot \bar{x}\}&=0 &\quad&\forall i\in I\label{knitro:cs}
\end{alignat}
\end{subequations}
plus ``the appropriate sign of the multiplier [$y_i$] depends on which bound (if either) is binding (active) at the solution.''. In particular, note that when a constraint is inactive, \eqref{knitro:cs} has the desired effect of forcing $y_i=0$. However, if a constraint is active \eqref{knitro:cs} does not enforce the desired sign constraint on $y_i$, which is why the additional comment in words is needed. Unfortunately, deducing the correct sign for $y_i$ may not be straightforward for someone unfamiliar with optimization theory.

We can check that \eqref{knitro:pf} and \eqref{knitro:stat} are our primal feasibility and dual stationarity conditions from Section~\ref{twosideddec}. In addition, the duals in Section~\ref{twosideddec} are sign-unrestricted and dual feasibility is always satisfied. Hence, the only difference with our optimality conditions is between \eqref{knitro:cs} and our complementary slackness condition \eqref{signed:cs}, which is $\min\{l_i\bar{y}_i,u_i\bar{y}_i\}=\bar{y}_i a^i\cdot \bar{x}$ for all $i\in I$. We can also check that \eqref{signed:cs} not only imposes $y_i=0$ when the associated constraint is inactive, but it also imposes the correct sign when the constraint is active (e.g. if $a_i\cdot\bar{x}=l_i$, then \eqref{signed:cs} imposes $\bar{y}_i\geq0$ through the assumption $l_i\leq u_i$).

\subsection{Domain Driven Dual}\label{dds:section}

 The primal problem from \cite{karimi2020primal,karimi2024domain} can be written in our primal form \eqref{primalproblem} by letting $I=\{1\}$, $A\in \mathbb{R}^{M_1\times J}$, $F^1:
 \mathbb{R}^J\to \mathbb{R}^{M_1}$ be given by $F^1(x):=Ax$, $D\subseteq \mathbb{R}^{M_1}$ be a closed convex set, and $g_1:
 \mathbb{R}^{M_1}\to \mathbb{R}\cup\{+\infty\}$ be such that $g_1=\delta_{D}$. The domain driven dual of this problem is 
\begin{subequations}\label{DDS:dual}
\begin{alignat}{3}
\operatorname{opt}_{DD}:=&\inf_{y\in \mathbb{R}^{M_1}}&   \sigma_D(y)\\
&s.t.\notag&&\\
&&A^T y&= -c,
\;y\in D_*,&
\end{alignat}
\end{subequations}
where  $D_\infty$ is the recession cone of $D$ and $D_*:=\{y\in \mathbb{R}^{M_1}\,:\, y\cdot w\leq 0\quad \forall w\in D_\infty\}$. Using \cite[Proposition C.2.2.4]{hiriart2004fundamentals} we can conclude that $y\in D_*$ is redundant in \eqref{DDS:dual}. Then using the fact that $\delta_D^*=\sigma_D$ we can conclude that \eqref{DDS:dual} is equivalent to the corresponding realization of  
\eqref{dualproblem} with $-y$ replaced by $y$ and with $\operatorname{opt}_{DD}=-\operatorname{opt}_{D}$

\section{Proofs}\label{proofs:sec}
\subsection{Proofs from Section~\ref{dualsection} and ~\ref{rayssection}}\label{mainpropertiesproofssection}

To show the results from Section~\ref{dualsection} and ~\ref{rayssection} we use the following corollary of the connection between $g_i\circ F^i$ and $\mathcal{O}_{F^i,g_i}$ discussed in Section~\ref{connectiontofenchelsection}.

\begin{restatable}{corollary}{Ofenchel}\label{Ofenchel}
Let $F^i:\mathbb{R}^J\to \mathbb{R}^{M_i}$ and $g_i:\mathbb{R}^{M_i} \to {\mathbb{R}\cup\{+\infty\}}$  satisfy \eqref{assumptioneq}. Then, for any  $\bar{w}^i,\tilde{w}^i\in \mathbb{R}^J$ and  ${y}^i\in \mathbb{R}^{M_i}$ such that $ \nabla F^i\left(\bar{w}^i\right)^T {y}^i=\nabla F^i\left(\tilde{w}^i\right)^T {y}^i$,
 \begin{equation}\label{OfenchelOeq}
 F^i_\nabla({y}^i, \bar{w}^i)=F^i_\nabla({y}^i, \tilde{w}^i).\end{equation}

Furthermore, for any $ (\bar{x},\bar{w}^i,{y}^i)\in \mathbb{R}^J \times\mathbb{R}^J \times \mathbb{R}^{M_i}$ we have
\begin{equation}\label{Ofencheleq}
    g_i\circ F^i(\bar{x})+\mathcal{O}_{F^i,g_i}\left({y}^i, \bar{w}^i\right)\geq \left(\nabla F^i\left(\bar{w}^i\right)^T {y}^i\right)\cdot \bar{x},
\end{equation}
and for any  $(\bar{x},{y}^i)\in \mathbb{R}^J \times \mathbb{R}^{M_i}$ there exists $\bar{w}^i\in \mathbb{R}^J$ such that  \eqref{Ofencheleq} holds at equality if and only if \eqref{Ofencheleq} holds at equality for $\bar{w}^i=\bar{x}$.

Finally, for any $\hat{x}\in \mathbb{R}^J$
\begin{equation}\label{Ofenchelhorizon}
    \left(g_i\circ F^i\right)_\infty(\hat{x})\geq \sup_{w\in \mathbb{R}^J}\sigma_{\operatorname{dom}(g^*_i)}\left(\nabla F^i(w)\hat{x}\right)
\end{equation}
and for any $ (\hat{w}^i,{y}^i)\in \mathbb{R}^J \times \mathbb{R}^{M_i}$ and $S_i:=\operatorname{dom}\left(g_i\circ F^i\right)$
\begin{equation}\label{Ofencheldomain}
\sigma_{S_i}\left(\nabla F^i\left(\hat{w}^i\right)^T {y}^i\right)\leq 
F^i_\nabla\left({y}^i, \hat{w}^i\right)+\left(g^*_i\right)_\infty\left({y}^i\right)
\end{equation}
\end{restatable}
\pointtoproof{Ofenchel:proof}

\ifrepeattheo
\weakduality*
\fi
\begin{proof}{\textbf{Proof of Theorem~\ref{weakduality}}\hspace*{0.5em}}
\label{weakduality:proof}
Summing \eqref{Ofencheleq} from Corollary~\ref{Ofenchel} for all $i \in I$  (with ${y}^i=-\bar{y}^i$) we get
 $\sum_{i\in I} g_i\left(F^i\left(\bar{x}\right)\right)+ \left(\nabla F^i\left(\bar{w}^i\right)^T \bar{y}^i\right)\cdot x\geq -\sum_{i\in I} \mathcal{O}_{F^i,g_i}(-\bar{y}^i, \bar{w})$. 
The result  follows by   \eqref{stationarity}.\Halmos
\end{proof}

\ifrepeattheo
\optimalityconditions*
\fi
\begin{proof}{\textbf{Proof of Theorem~\ref{optimalityconditions}}\hspace*{0.5em}}
\label{optimalityconditions:proof}
If $(\bar{y},\bar{w})$ satisfies  \eqref{stationarity} then 
$c\cdot \bar{x}+\sum_{i\in I} g_i\left(F^i\left(\bar{x}\right)\right) +\sum_{i \in I}   \mathcal{O}_{F^i,g_i}(-\bar{y}^i, w^i)=\sum_{i\in I} \alpha_i$
for $\alpha_i:= g_i\left(F^i\left(\bar{x}\right)\right)+\mathcal{O}_{F^i,g_i}(-\bar{y}^i, \bar{w}^i)-\left(\nabla F^i\left(\bar{w}^i\right)^T (-\bar{y}^i)\right)\cdot\bar{x}$.
By Corollary~\ref{Ofenchel} (with ${y}^i=-\bar{y}^i$), $\alpha_i\geq0$ for all $i\in I$ and by Theorem~\ref{weakduality} $\sum_{i\in I}\alpha_i\geq0$. Then $\sum_{i\in I}\alpha_i=0$ if and only if $\alpha_i=0$ for all $i\in I$ and the result follows. Finally, by Corollary~\ref{Ofenchel} (with ${y}^i=-\bar{y}^i$) if \eqref{optimaliticonditioneq} holds for some $(\bar{x},\bar{y}^i,\bar{w}^i)\in \mathbb{R}^J\times \mathbb{R}^{M_i}\times \mathbb{R}^J$, then it also holds for $(\bar{x},\bar{y}^i,\bar{x})$, which shows the last statement.\Halmos
\end{proof}

\ifrepeattheo
\csoptimalitycorollary*
\fi
\begin{proof}{\textbf{Proof of Corollary~\ref{csoptimalitycorollary}}\hspace*{0.5em}}
\label{csoptimalitycorollary:proof}
If $\bar{x}$ is primal feasible, then $\operatorname{obj_P}(\bar{x})=c\cdot \bar{x}+\sum_{i\in I} g_i\left(F^i\left(\bar{x}\right)\right)$. Similarly, if $(\bar{y}^i,\bar{w}^i)$ is dual feasible for each $i\in I$, then $\mathcal{O}_{F^i,g_i}^0(-\bar{y}^i, \bar{w}^i)=\mathcal{O}_{F^i,g_i}(-\bar{y}^i, \bar{w}^i)$ for each $i\in I$ and $\operatorname{obj_D}(\bar{y},\bar{w})=-\sum_{i \in I} \mathcal{O}_{F^i,g_i}(-\bar{y}^i, \bar{w}^i)$. Then the complementary slackness based optimality conditions are equivalent to \eqref{stationarity} and \eqref{optimaliticonditioneq} for all $i\in I$ and the results follows from Theorem~\ref{optimalityconditions}. \Halmos
\end{proof}

\ifrepeattheo
\equalobjlemma*
\fi
\begin{proof}{\textbf{Proof of Lemma~\ref{equalobjlemma}}\hspace*{0.5em}}
\label{equalobjlemma:proof}
The result follows by summing \eqref{cscondition} for all $i\in I$ and applying \eqref{stationarity}. \Halmos
\end{proof}

\ifrepeattheo
\unboundedtheo*
\fi
\begin{proof}{\textbf{Proof of Theorem~\ref{unboundedtheo}}\hspace*{0.5em}}
\label{unboundedtheo:proof}
For the primal case, let $h(x)=c\cdot x+\sum_{i\in I} g_i\left(F^i(x)\right)$. If $\bar{x}\in\mathbb{R}^J$ is a primal feasible solution such that $h(\bar{x})<+\infty$, then by  
\cite[Proposition B.3.2.8]{hiriart2004fundamentals} we have that $h_\infty(d)=c\cdot d+\sum_{i\in I} (g_i\circ F^i)_\infty(d)$. In addition,
$h(\bar{x}+t \hat{x})\leq h(\bar{x})+t h_\infty(\hat{x})$ 
for all $t>0$. Then if $\hat{x}$ is a primal ray the result follows by letting $t$ go to infinity.

For the dual case, first note that by \eqref{OfenchelOeq} in Corollary~\ref{Ofenchel} (with ${y}^i=-\bar{y}^i$) and \eqref{rayequalnabla} we have that $(\hat{y},\bar{w})$ is also a dual ray so we may assume $\bar{w}=\hat{w}$. Let $h= \sum_{i \in I}   \mathcal{O}_{F^i,g_i}(-y^i, \bar{w}^i)= \sum_{i \in I}   \mathcal{O}_{F^i,g_i}(-y^i, \hat{w}^i)$. Because $h(y)$ is proper, closed and convex and $h(\bar{y})<+\infty$, by  \cite[Proposition B.3.2.8]{hiriart2004fundamentals} we have that $h_\infty(\hat{y})=\sum_{i\in I} \left(\nabla F^i\left(\hat{w}^i\right)\hat{w}^i -  F^i\left(\hat{w}^i\right)\right)\cdot {\hat{y}^i}+(g_i^*)_\infty\left(\hat{y}^i\right)$. In addition,
$h(\bar{y}+t \hat{y})\leq h(\bar{y})+t h_\infty(\hat{y})$
for all $t>0$. Finally, $(\bar{y}+t \hat{y},\bar{w})=(\bar{y}+t \hat{y},\hat{w})$ satisfies \eqref{stationarity} for all $t$. The result then follows by letting $t$ go to infinity.\Halmos
\end{proof}

\ifrepeattheo
\infeasibilitytheo*
\fi
\begin{proof}{\textbf{Proof of Theorem~\ref{infeasibilitytheo}}}
\label{infeasibilitytheo:proof}
Let $\hat{x}\in \mathbb{R}^J$ be a primal ray and for each $w\in \mathbb{R}^{I\times J}$ let $v(w)=\left(\nabla F^i\left(w^i\right)\hat{x}\right)_{i\in I}\in \mathbb{R}^M$, $S(w)=\{y\in \mathbb{R}^M\,:\, \sum_{i\in I} \nabla F^i\left(w^i\right)^{T}y^i= c\}$ and 
$C=\left\{y\in \mathbb{R}^M\,:\, -y^i\in \operatorname{dom}(g_i^*)\quad \forall i\in I \right\}$. We have $\sigma_{S(w)}(v(w))=c\cdot\hat{x}$ when $S(w)\neq \emptyset$, and $\sigma_{S(w)}(v(w))=-\infty$ when $S(w)= \emptyset$.
In addition,  
\begin{align*}
c\cdot \hat{x} &<-\sum_{i\in I} (g\circ F)_\infty(\hat{x})\\
&\leq -\sum_{i\in I} \sup_{w^i\in\mathbb{R}^J}\sigma_{\operatorname{dom}(g_i^*)}(\nabla F^i\left(w^i\right)\hat{x})\\
&=-\sup_{w\in\mathbb{R}^{I\times J}}\sigma_{C}\left(-v(w)\right)=\inf_{w\in\mathbb{R}^{I\times J},y\in C} v(w)\cdot y
\end{align*}
where the first inequality follows from \eqref{primalraydef} and the second inequality follows from \eqref{Ofenchelhorizon} in Corollary~\ref{Ofenchel}. Then, for any $w\in \mathbb{R}^{I\times J}$ such that $S(w)\neq \emptyset$ we have $\sup\{v(w)\cdot y\,:\, y\in S(w)\}< \inf\{v(w)\cdot y\,:\, y\in C\}$,
so $S(w)\cap C=\emptyset$ for all $w\in \mathbb{R}^{I\times J}$  and the dual problem is infeasible.

Let $(\hat{y},\hat{w})\in \mathbb{R}^{M}\times \mathbb{R}^{I\times J}$ be a dual ray and for each $i\in I$ let $S_i:=\operatorname{dom}\left(g_i\circ F^i\right)$ and $v^i:=\nabla F^i(\hat{w}^i)^T \hat{y}^i$. By \eqref{Ofencheldomain} from Corollary~\ref{Ofenchel} (with ${y}^i=-\hat{y}^i$) for all $i\in I$ we have 
\begin{equation}\label{Ofencheldomainrestate}
\sigma_{S_i}\left(-v^i\right)\leq 
F^i_\nabla\left(-\hat{y}^i, \hat{w}^i\right)+\left(g^*_i\right)_\infty\left(-\hat{y}^i\right).
\end{equation}
Let's assume for a contradiction that $\bigcap_{i\in I} S_i\neq \emptyset$ and let $k\in I$ so that $\bigcap_{i\in I\setminus\{k\}} S_i\neq \emptyset$. Noting that condition \eqref{dualrayeq} for a dual ray is equivalent to $v^k=\sum_{i\in I\setminus\{k\}} -v^i$, we have that  \cite[Theorem 11.24(e)]{rockafellar2009variational} implies
\begin{equation}\label{Ofencheldualsupportcalculus}\sigma_{\bigcap_{i\in I\setminus\{k\}}S_i}\left(v^k\right)\leq
\sum_{i\in I\setminus\{k\}} \sigma_{S_i}(-v^i).
\end{equation}
Combining \eqref{Ofencheldualsupportcalculus} with \eqref{Ofencheldomainrestate} for $i\in I\setminus\{k\}$ and dual ray condition 
\eqref{dualrayineq} we get
\begin{equation}\label{Ofencheldualalmostlast}
\sigma_{\bigcap_{i\in I\setminus\{k\}}S_i}\left(v^k\right)< 
-F^k_\nabla\left(-\hat{y}^k, \hat{w}^k\right)-\left(g^*_k\right)_\infty\left(-\hat{y}^k\right).
\end{equation}
Combining \eqref{Ofencheldualalmostlast} with \eqref{Ofencheldomainrestate} for $i=k$ we get $
\sigma_{\bigcap_{i\in I\setminus\{k\}}S_i}\left(v^k\right)< 
-\sigma_{S_k}\left(-v^k\right)$ or equivalently $
\sup\{v^k\cdot x\,:\, x\in\bigcap_{i\in I\setminus\{k\}}S_i\}< 
\inf\{v^k\cdot x\,:\,x\in S_k\}$, which contradicts $\bigcap_{i\in I} S_i\neq \emptyset$. \Halmos
\end{proof}

\section{Proofs from Section~\ref{secificssection}}

\ifrepeattheo
\conicspecific*
\fi
\begin{proof}{\textbf{Proof of Proposition~\ref{conicspecific}}}
\label{conicspecific:proof}
The characterization of $g_i^*$ follows from known properties of the conjugate of an indicator function (e.g. \cite[Example E.1.1.5]{hiriart2004fundamentals}). The characterization of the asymptotic functions follows from \cite[Proposition B.3.2.1, Proposition B.3.2.8]{hiriart2004fundamentals} and the fact that the epigraph of indicator of a closed convex cone is also a closed convex cone and the recession cone of a closed convex cone is itself. The other characterizations are direct.\Halmos
\end{proof}

\ifrepeattheo
\twosidedspecific*
\fi
\begin{proof}{\textbf{Proof of Proposition~\ref{twosidedspecific}}}
\label{twosidedspecific:proof}
The characterization of $g_i^*$ follows from known properties of the conjugate of an indicator function (e.g. \cite[Example E.1.1.5]{hiriart2004fundamentals}). The characterization of the asymptotic functions follows from \cite[Proposition B.3.2.1, Proposition B.3.2.8]{hiriart2004fundamentals} and the fact that the epigraph of any support function is a closed convex cone, that the recession cone of a closed convex cone is itself, and the fact that $\delta_{[l_i,u_i]}^{
\infty}=\delta_{
\{0\}}$. The other characterizations are direct.\Halmos
\end{proof}

\ifrepeattheo
\simplequadraticobjlemma*
\fi
\begin{proof}{\textbf{Proof of Proposition~\ref{simplequadraticobjlemma}}}
\label{simplequadraticobjlemma:proof}
The characterization of $g^*$  is direct from the definition of the convex conjugate. The characterization of the asymptotic functions
follows from \cite[Proposition B.3.2.1]{hiriart2004fundamentals} and simple properties of the epigraph of a convex quadratic form. The characterization of $
\nabla F^i$ is direct and that of $F^i_\Delta$ follows from the fact that $\nabla h_i\left(w^i\right)\cdot w^i=2h_i\left(w^i\right)$ by \cref{quadraticform:lemma}.
\Halmos
\end{proof}

\ifrepeattheo
\quadraticconstrlemma*
\fi
\begin{proof}{\textbf{Proof of Proposition~\ref{quadraticconstrlemma}}}
\label{quadraticconstrlemma:proof}
The characterization of $g_i^*$ follows from known properties of the conjugate of an indicator function (e.g. \cite[Example E.1.1.5]{hiriart2004fundamentals}). The characterization of the asymptotic functions follows from \cite[Proposition B.3.2.1]{hiriart2004fundamentals}, the fact that the epigraph of indicator of a closed convex cone is also a closed convex cone, that the recession cone of a closed convex cone is itself and simple properties of the level sets of quadratic functions. The characterization of $
\nabla F^i$ is direct and that of $F^i_\Delta$ follows from the fact that $\nabla h_i\left(w^i\right)\cdot w^i=2h_i\left(w^i\right)$ for any quadratic form $h_i$.\Halmos
\end{proof}

\ifrepeattheo
\quadraticformulationlemma*
\fi
\begin{proof}{\textbf{Proof of Proposition~\ref{quadraticformulationlemma}}}
\label{quadraticformulationlemma:proof}
Follows from simple case analysis and the fact that $ h_i\left(y_i w^i\right)=y_i^2 h_i\left( w^i\right)$ for any quadratic form $h_i$.\Halmos
\end{proof}

\section{Proofs from Section~\ref{connectiontodualssection}}

We use Theorem 3.1 and Lemma 3.2 from \cite{vielma25}, which when specialized to our setting can be written as follows (see also \cite{burke2021study,boct2007new,boct2008new,ioan2009generalized,combari1994sous,gissler2023note,pennanen1999graph}).

\begin{theorem}[\cite{vielma25}]\label{simplecomposite}
If  $F^i:\mathbb{R}^J\to \mathbb{R}^{M_i}$ and $g_i:\mathbb{R}^{M_i} \to \mathbb{R}\cup \{+\infty\}$ satisfy \eqref{assumptioneq}, 
then for any $\bar{v}^i\in \mathbb{R}^J$,
\begin{equation}\label{connectiontofenchelfirstmainineq}
    \left(g_i \circ F^i\right)^*\left(\bar{v}^i\right)\leq \inf_{y^i \in Y^i}\left(y^i\cdot F^i\right)^*\left(\bar{v}^i\right)+g_i^*\left(y^i\right),
\end{equation}
where $Y^i$ can equivalently taken to be equal to $\mathbb{R}^{M_i}$ or $\operatorname{dom}(g^*_i)$. 

If $g_i\circ F^i$ also satisfies constraint qualification \eqref{constraintqual}, then  \eqref{connectiontofenchelfirstmainineq} holds at equality and the infimum in  \eqref{connectiontofenchelfirstmainineq} is a minimum (possibly with $\left(g_i \circ F^i\right)^*\left(\bar{v}^i\right)=\left(y^i\cdot F^i\right)^*\left(\bar{v}^i\right)+g_i^*\left(y^i\right)=+\infty$ for all $y^i\in \mathbb{R}^{M_i}$).

Finally, $\bar{y}^i\cdot F^i(\cdot)$ is convex for all $\bar{y}^i\in \operatorname{dom}(g^*_i)$.
\end{theorem}
\begin{proof}{\textbf{Proof}\hspace*{0.5em}}
The statements on $\left(g_i \circ F^i\right)^*$ follow from \cite[Theorem 3.1]{vielma25}. The statement on $\bar{y}^i\cdot F^i(\cdot)$ follows from \cite[Lemma 3.2]{vielma25}.
\Halmos
\end{proof}

We also use the following properties of the classical Legendre transform (e.g. see \cite[Section 26]{rockafellar2015convex} and \cite[Section 11.C]{rockafellar2009variational}).
\begin{proposition}\label{differentiableobjlemma}
Let $h:\mathbb{R}^J\to \mathbb{R}$ be a convex differentiable function and $\mathcal{L}_h:\mathbb{R}^J\to \mathbb{R}\cup\{+\infty\}$ be given by
\[\mathcal{L}_h(v)=\begin{cases}\nabla h(w)\cdot w-h(w)&\nabla h(w)=v\\+\infty&\text{ o.w. }\end{cases}.\]
Then, $\mathcal{L}_h$ is well defined (i.e. we get the same value for any $w\in \mathbb{R}^J$ such that $\nabla h(w)=v$) and 
\begin{equation}\label{differentiableobjlemmaineq}
h^*(v) \leq \mathcal{L}_h(v)
\end{equation}
for all $v\in \mathbb{R}^J$. In addition,  for any $x,w,v\in \mathbb{R}^J$ we have
\begin{subequations}\label{differentiableobjlemmaequiv}
\begin{align}
    h^*(\nabla h(w))&=\mathcal{L}_h(\nabla h(w)), \label{differentiableobjlemmaequiv:a}\\
    h(x)+h^*(v)&=v\cdot x  \quad\Leftrightarrow\quad  v=\nabla h(x).\label{differentiableobjlemmaequiv:b}
\end{align}  
\end{subequations}
In particular, \eqref{differentiableobjlemmaineq} holds at equality if there exist $\bar{w}\in \mathbb{R}^J$ such that $\nabla h(\bar{w})=v$ or if $h^*(v)=+\infty$.
\end{proposition}
\begin{proof}{\textbf{Proof}\hspace*{0.5em}}
Condition $\nabla h(\bar{w})=v$ is equivalent to the optimal value of convex optimization problem  
$h^*(v)=\sup_{w}\left\{v\cdot w-h(w)\right\}$ being finite and attained by $\bar{w}$. When such $\bar{w}$ exists, the optimal value is  $\nabla h(\bar{w})\cdot \bar{w}-h(\bar{w})$, which shows that $\mathcal{L}_h$ is well defined and inequality \eqref{differentiableobjlemmaineq} (with strict inequality when the optimal value in $h^*(w)$ is finite, but not attained). Equation \eqref{differentiableobjlemmaequiv} then follows directly.\Halmos
\end{proof}

\ifrepeattheo
\connectiontofenchelcoro*
\fi
\begin{proof}{\textbf{Proof of Corollary~\ref{connectiontofenchelcoro}}\hspace*{0.5em}}
\label{connectiontofenchelcoro:proof}
Let $\bar{y}^i\in \operatorname{dom}(g^*_i)$ and
 $h=\bar{y}^i\cdot F^i$. Then $h$ is differentiable and by Theorem~\ref{simplecomposite} it is also convex. Then \eqref{connectiontofenchelcoroineqopt} is a convex optimization problem. In addition, $\nabla h\left(w^i\right)=\nabla F^i\left(w^i\right)^T\bar{y}^i$ so $\nabla h\left(w^i\right)\cdot w^i-h\left(w^i\right)=F^i_\nabla\left(\bar{y}^i, w^i\right)$ and hence $\mathcal{L}_h(-\bar{v}^i)=\inf_{w^i\in\mathbb{R}^J}\left\{F^i_\nabla\left(\bar{y}^i, w^i\right)\,:\, \nabla F^i\left(w^i\right)^T\bar{y}^i=-\bar{v}^i\right\}$ . Then, by  Proposition~\ref{differentiableobjlemma} we have $-\left(\bar{y}^i\cdot F^i\right)^*\left(-\bar{v}^i\right)\geq\sup_{w^i\in\mathbb{R}^J}\left\{-F^i_\nabla\left(\bar{y}^i, w^i\right)\,:\, \nabla F^i\left(w^i\right)^T\bar{y}^i=-\bar{v}^i\right\}$, with equality and with $\sup$ replaced by $\max$ if there exist $\bar{w}^i\in \mathbb{R}^J$ such that $\nabla F^i\left(\bar{w}^i\right)^T\bar{y}^i=-\bar{v}^i$.

By Theorem~\ref{simplecomposite}, 
$ -\left(g_i \circ F^i\right)^*\left(-\bar{v}^i\right)\geq \sup_{y^i \in \operatorname{dom}(g^*_i)}-\left(y^i\cdot F^i\right)^*\left(-\bar{v}^i\right)-g_i^*\left(y^i\right)$ with equality and with $\sup$ replaced by $\max$ if $g_i\circ F^i$ satisfies  \eqref{constraintqual}. Combining this with results of the previous  paragraph we get
$ -\left(g_i \circ F^i\right)^*\left(-\bar{v}^i\right)\geq \sup_{y^i \in \operatorname{dom}(g^*_i),w^i\in\mathbb{R}^J}\left\{-F^i_\nabla\left({y}^i, w^i\right)-g_i^*\left(y^i\right)\,:\, \nabla F^i\left(w^i\right)^T{y}^i=-\bar{v}^i\right\}$ or equivalently $ -\left(g_i \circ F^i\right)^*\left(-\bar{v}^i\right)\geq \sup_{y^i \in \mathbb{R}^{M_i},w^i\in\mathbb{R}^J}\left\{-F^i_\nabla\left(-{y}^i, w^i\right)-g_i^*\left(-y^i\right)\,:\, \nabla F^i\left(w^i\right)^T{y}^i=\bar{v}^i\right\}$, with equality and with $\sup$ replaced by $\max$ if $g_i\circ F^i$ satisfies  \eqref{constraintqual} and for all $\bar{y}^i \in \operatorname{dom}(g^*_i)$ there exist $\bar{w}^i\in \mathbb{R}^J$ such that $\nabla F^i\left(\bar{w}^i\right)^T\bar{y}^i=-\bar{v}^i$ whenever $-\left(\bar{y}^i\cdot F^i\right)^*\left(-\bar{v}^i\right)>-\infty$. The result follows by noting that this last condition for equality is equivalent to the optimal value  in \eqref{connectiontofenchelcoroineqopt} being attained whenever it is finite (i.e. condition \eqref{connectiontofenchelcoroineqopt:cond}).\Halmos
\end{proof}

To prove Corollary~\ref{Ofenchel} we also need the following lemma. 
\begin{lemma}\label{horizonineqlemma}
Let $f^i:\mathbb{R}^{n} \to {\mathbb{R}\cup\{+\infty\}}$ for $i\in \{1,2\}$ be proper closed convex functions such that $f_1\leq f_2$. Then $(f_1)_\infty\leq (f_2)_\infty$.
\end{lemma}
\begin{proof}{\textbf{Proof}\hspace*{0.5em}}
If $f_1\leq f_2$, then $f_2^*\leq f_1^*$,  $\operatorname{dom}(f_1^*)\subseteq \operatorname{dom}(f_2^*)$ and $\sigma_{\operatorname{dom}(f_1^*)}\subseteq \sigma_{\operatorname{dom}(f_2^*)}$. The result follows from 
\cite[Proposition 11.5]{rockafellar2009variational}, which states that $\sigma_{\operatorname{dom}(f^*)}=f_\infty$ for any proper closed convex function $f$.
\Halmos
\end{proof}

\ifrepeattheo
\Ofenchel*
\fi
\begin{proof}{\textbf{Proof of Corollary~\ref{Ofenchel}}\hspace*{0.5em}}
\label{Ofenchel:proof}
For ${y}^i\in \mathbb{R}^{M_i}$, let $h:={y}^i\cdot F^i$ so that $\nabla h\left({w}^i\right)=\nabla F^i\left({w}^i\right)^T{y}^i$ and $\mathcal{L}_{h}\left(\nabla h\left({w}^i\right)\right)=F^i_\nabla\left({y}^i, {w}^i\right)$ for any $w^i\in \mathbb{R}^{M_i}$. Then \eqref{OfenchelOeq} follows because $\mathcal{L}_{h}$ is well defined by Proposition~\ref{differentiableobjlemma}.

Inequality \eqref{Ofencheleq} for $ (\bar{x},\bar{w}^i,{y}^i)\in \mathbb{R}^J \times\mathbb{R}^J \times \mathbb{R}^{M_i}$ follows from
\begin{subequations}\label{Ofenchel:ineqs}
\begin{align}
\label{Ofenchel:ineqs:a}
    \nabla h\left(\bar{w}^i\right)\cdot \bar{x}-g_i\circ F^i(\bar{x})&\leq \left(g_i\circ F^i\right)^*\left(\nabla h\left(\bar{w}^i\right)\right)\\
    \label{Ofenchel:ineqs:b}
    &\leq h^*\left(\nabla h\left(\bar{w}^i\right)\right) + g_i^*\left({y}^i\right)\\
    \label{Ofenchel:ineqs:c}
    &= \mathcal{L}_{h}\left(\nabla h\left(\bar{w}^i\right)\right)+ g_i^*\left({y}^i\right)\\
    &=\mathcal{O}_{F^i,g_i}\left({y}^i, \bar{w}^i\right),
\end{align}
\end{subequations}
where \eqref{Ofenchel:ineqs:a} follows from Fenchel's inequality (e.g. \cite[Proposition 11.3]{rockafellar2009variational}) for $g_i\circ F^i$, \eqref{Ofenchel:ineqs:b} follows from  \eqref{connectiontofenchelfirstmainineq} in Theorem~\ref{simplecomposite} and \eqref{Ofenchel:ineqs:c} follows from \eqref{differentiableobjlemmaequiv:a} in Proposition~\ref{differentiableobjlemma}.

If  \eqref{Ofencheleq} holds at equality, then all the inequalities in \eqref{Ofenchel:ineqs} also hold at equality, so in particular 
\begin{subequations}
\begin{align}
\label{Ofenchel:ineqs2:a}
    \nabla h\left(\bar{w}^i\right)\cdot \bar{x}-h^*\left(\nabla h\left(\bar{w}^i\right)\right)&=  g_i\left(F^i(\bar{x})\right)+g_i^*\left({y}^i\right)\\
    \label{Ofenchel:ineqs2:c}
    &\geq {y}^i\cdot F^i(\bar{x})=h(\bar{x}),
\end{align}
\end{subequations}
where \eqref{Ofenchel:ineqs2:a} follows from equality between the left hand side of \eqref{Ofenchel:ineqs:a} and the right hand side of \eqref{Ofenchel:ineqs:b}, and \eqref{Ofenchel:ineqs2:c} follows from Fenchel's inequality for $g_i$. Then $h(\bar{x})+h^*\left(\nabla h\left(\bar{w}^i\right)\right)=\nabla h\left(\bar{w}^i\right)\cdot \bar{x}$, which by \eqref{differentiableobjlemmaequiv:b} in Proposition~\ref{differentiableobjlemma} (for $v=\nabla h\left(\bar{w}^i\right)$) implies 
$\nabla F^i\left(\bar{w}^i\right)^T{y}^i=\nabla h\left(\bar{w}^i\right)=\nabla h\left(\bar{x}\right)=\nabla F^i\left(\bar{x}\right)^T{y}^i$.
Then, by \eqref{OfenchelOeq} we also have $F^i_\nabla\left({y}^i, \bar{w}^i\right)=F^i_\nabla\left({y}^i,\bar{x}\right)$,  $\mathcal{O}_{F^i,g_i}\left({y}^i, \bar{w}^i\right)=\mathcal{O}_{F^i,g_i}\left({y}^i, \bar{x}\right)$ and hence  \eqref{Ofencheleq} also holds at equality for $\bar{w}^i=\bar{x}$.

For \eqref{Ofenchelhorizon} note that by \eqref{connectiontofenchelcoroineq} in  Corollary~\ref{connectiontofenchelcoro}
$ \operatorname{dom}\left(\left(g_i \circ F^i\right)^*\right)\supseteq \bigcup_{y^i\in \operatorname{dom}\left(g_i^*\right),w^i\in  \mathbb{R}^J} \left\{\nabla F^i\left(w^i\right)^T y^i\right\}$.
Then $\sigma_{\operatorname{dom}\left(\left(g_i \circ F^i\right)^*\right)}\left(\hat{x}\right)\geq \sup_{w^i\in  \mathbb{R}^J} \sigma_{\operatorname{dom}\left(g_i^*\right)}\left(\nabla F^i\left(w^i\right)\hat{x}\right)$
and the result follows because $\sigma_{\operatorname{dom}\left(\left(g_i \circ F^i\right)^*\right)}= \left(g_i\circ F^i\right)_\infty$ by 
\cite[Theorem 11.5]{rockafellar2009variational}.

For \eqref{Ofencheldomain} let $h^1_{w^i}\left(y^i\right)=\left(g_i\circ F^i\right)^*\left(\nabla F^i({w^i})^T {y^i}\right)$ and $h^2_{w^i}\left(y^i\right)=\mathcal{O}_{F^i,g_i}\left({y^i}, {w^i}\right)$. Then by \eqref{connectiontofenchelcoroineq} in  Corollary~\ref{connectiontofenchelcoro}
$h_{w^i}^1\left(y^i\right)\leq h_{w^i}^2\left(y^i\right)$
for all $w^i\in \mathbb{R}^J$ and $y^i\in \mathbb{R}^{M_i}$. In addition $h_{w^i}^1$ and $h_{w^i}^1$ are proper closed convex functions so by Lemma~\ref{horizonineqlemma} we also have
$\left(h_{w^i}^1\right)_\infty\left(y^i\right)\leq\left(h_{w^i}^2\right)_\infty\left(y^i\right)$.
The result follows by noting that by \cite[Proposition B.3.2.8]{hiriart2004fundamentals} we have  $\left(h_{w^i}^2\right)_\infty\left(y^i\right)=F^i_\nabla\left({y}^i, {w}^i\right)+\left(g^*_i\right)_\infty\left({y}^i\right)$ and $(h_{w^i}^1)_\infty\left(y^i\right)=\left(\left(g_i\circ F^i\right)^*\right)_\infty\left(\nabla F({w^i})^T {y^i}\right)$, and that by \cite[Theorem 11.5]{rockafellar2009variational} $\left(\left(g_i\circ F^i\right)^*\right)_\infty\left(\nabla F({w^i})^T {y^i}\right)=\sigma_{\operatorname{dom}\left(g_i\circ F^i\right)}\left(\nabla F^i\left({w}^i\right)^T {y}^i\right)$.\Halmos
\end{proof}

\ifrepeattheo
\attainmentlemma*
\fi
\begin{proof}{\textbf{Proof of Lemma~\ref{attainmentlemma}}}
\label{attainmentlemma:proof}
For the functions in Table~\ref{initialfunctions} we have that $F^i$ is an affine or convex quadratic function. Then the objective function $\bar{v}^i\cdot w^i-\bar{y}^i\cdot F^i\left(w^i\right)$ of \eqref{connectiontofenchelcoroineqopt} is an affine or concave quadratic function whose unconstrained suppremum is always attained when finite.
\Halmos
\end{proof}

\ACKNOWLEDGMENT{The authors would like to thank Robert Fourer for pointing out the connection between Rockafellar's monotropic programming duality \cite{rockafellar1981monotropic} and the piecewise linear dual \eqref{standard:ub:dual:lp} for LPs with lower and upper variable bounds (e.g \cite{fourer94}).}

\bibliographystyle{informs2014}
\bibliography{references}
\renewcommand{\theHsection}{A\arabic{section}}
\begin{APPENDICES}
\section{Changes for maximization problems}\label{max:sec}
A maximization version of primal problem  \eqref{primalproblem} is given by
\begin{equation}\label{primalproblem:max}
    \operatorname{opt}_P:=\sup_{x\in \mathbb{R}^J} \; c\cdot x - \sum_{i\in I} g_i\left(F^i(x)\right)
\end{equation}
and the associated version of dual problem \eqref{dualproblem} becomes
\begin{subequations}\label{dualproblem:max}
\begin{alignat}{3}
\operatorname{opt}_D:=&\inf_{y\in \mathbb{R}^M, w\in \mathbb{R}^{I\times J}}\quad&   \sum_{i \in I}    F^i_\nabla\left(y^i, w^i\right)&+g^*_i\left(y^i\right) &\\\
&s.t.&&\notag\\
&& \sum_{i\in I} \nabla F(w^i)^T y^i&= c. 
\end{alignat}
\end{subequations}
Primal objective \eqref{primalobjmin} becomes 
\[\operatorname{obj_P}(\bar{x}):=c\cdot \bar{x}-\sum_{i\in I} g_i^0\left(F^i\left(\bar{x}\right)\right)\]
and dual objective \eqref{dualobjectivemin} become
\[\operatorname{obj_D}(\bar{y},\bar{w}):=\sum_{i \in I} \mathcal{O}_{F^i,g_i}^0(\bar{y}^i, \bar{w}^i).\]

Weak duality inequalities \eqref{weakdualityone}/\eqref{weakdualitytwo} become 
\begin{align*}c\cdot \bar{x}-\sum_{i\in I} g_i\left(F^i\left(\bar{x}\right)\right)&\leq \sum_{i \in I} \mathcal{O}_{F^i,g_i}(\bar{y}^i, \bar{w}^i),\\
\operatorname{obj_P}(\bar{x})&\leq \operatorname{obj_D}(\bar{y},\bar{w}),
\end{align*}
optimality condition \eqref{optimaliticonditioneq} becomes
\begin{equation}\label{optimaliticonditioneqmax}g_i\left(F^i\left(\bar{x}\right)\right)+\mathcal{O}_{F^i,g_i}(\bar{y}^i, \bar{w}^i)=\left(\nabla F^i\left(\bar{w}^i\right)^T \bar{y}^i\right)\cdot\bar{x},\end{equation}
and complementary slackness condition \eqref{cscondition} becomes \begin{equation}\label{csconditionmax}
g_i^0\left(F^i\left(\bar{x}\right)\right)+\mathcal{O}^0_{F^i,g_i}(\bar{y}^i, \bar{w}^i)=\left(\nabla F^i\left(\bar{w}^i\right)^T \bar{y}^i\right)\cdot\bar{x}.
\end{equation}

Similarly the ray definitions become 
\begin{equation}\label{primalraydefmax}c\cdot \hat{x}-\sum_{i\in I} (g_i\circ F^i)_\infty(\hat{x})>0,\end{equation}
and dual ray definition \eqref{dualraydef} becomes
\begin{subequations}\label{dualraydefmax}
\begin{align}
\sum_{i\in I} \nabla F^i(\hat{w}^i)^T \hat{y}^i&= 0\label{dualrayeqmax}\\ 
    \sum_{i \in I} F^i_\nabla(\hat{y}^i, \hat{w}^i) +(g_i^*)_\infty\left(\hat{y}^i\right)&< 0.\label{dualrayineqmax}
\end{align}
 \end{subequations}

\subsection{Conic constraints (including one-sided linear inequalities)}\label{conicsectionmax}

Maximization primal problem \eqref{primalproblem:max} with all $F^i$ and $g_i$ as in Proposition~\ref{conicspecific} reduces to 
\begin{align*}
\operatorname{opt}_P:&=
\begin{alignedat}[t]{3}
&\sup_{x\in \mathbb{R}^J}& c\cdot x&\\
&s.t.&&\\
&& A^ix-b^i&\in K_i \quad \forall i\in I
\end{alignedat}
\end{align*}
and the associated dual problem  \eqref{dualproblem:max} reduces to
\begin{alignat*}{3}
\operatorname{opt}_D:=
&\inf_{y\in \mathbb{R}^M}&  \sum_{i \in I} b^i\cdot y^i  & \\\
&s.t.&&\notag\\
&& \sum_{i\in I} (A^i)^{T}y^i&= c,\quad y^i&\in K_i^\circ\quad \forall i \in I.
\end{alignat*}
As expected, because $F^i$ is an affine function, variables $w^i$ do not appear in this dual. 

Primal feasibility becomes the standard conic primal feasibility $A^ix-b^i\in K_i$ for all $i\in I$, and dual feasibility becomes the standard conic dual feasibility $\sum_{i\in I} (A^i)^{T}y^i= c$ and $y^i\in K_i^\circ$ for all $i\in I$. Complementary slackness is the standard conic complementary slackness 
$\bar{y}^i\cdot (A^i \bar{x} -b^i)=0$  for all  $i \in I$. In addition, primal ray definition \eqref{primalraydefmax} becomes the standard conic primal ray condition
\begin{equation*}A^i\hat{x}\in K_i\quad\forall i\in I \quad\text{and}\quad c\cdot \hat{x}>0,\end{equation*}
and dual ray definition \eqref{dualraydefmax} becomes the standard conic dual ray condition
\begin{align*}
    \sum_{i\in I} (A^i)^{T}\hat{y}^i &=0,\quad \hat{y}^i\in K_i^\circ\quad\forall i \in I \\ \sum_{i \in I} b^i\cdot \hat{y}^i &<0.
\end{align*}
Again, we can see that variables $w^i$ are also not needed to describe dual rays.

\subsection{Two-sided (or ranged) linear inequalities}

Primal problem \eqref{primalproblem:max} with all $F^i$ and $g_i$ as in Proposition~\ref{twosidedspecific} reduces to 
\begin{alignat*}{3}
\operatorname{opt}_P:=
&\sup_{x\in \mathbb{R}^J}& c\cdot x&\\
&s.t.&&\notag\\
&& l_i\leq a^i\cdot x&\leq u_i \quad \forall i \in I  
\end{alignat*}
and the associated dual problem  \eqref{dualproblem:max} reduces to
\begin{alignat*}{3}
\operatorname{opt}_D:=
&\inf_{y\in \mathbb{R}^M}& \quad \sum_{i \in I} \max\{l_i y_i,&u_i y_i\}    \\\
&s.t.&&\notag\\
&& \sum_{i\in I} a^iy_i&= c.
\end{alignat*}
As expected, because $F^i$ is an linear function, variables $w^i$ do not appear in this dual. 

Primal feasibility becomes  $l_i\leq a^i\cdot x\leq u_i$ for all $i\in I$, all module-wise dual solutions $\bar{y}_i$ are feasible, and feasibility of a full dual solution $\bar{y}$ reduces to satisfying stationarity condition $\sum_{i\in I} a^iy_i= c$. Complementary slackness becomes 
\begin{equation*}
    \max\{l_i\bar{y}_i,u_i\bar{y}_i\}=\bar{y}_i a^i\cdot \bar{x} \quad \forall i \in I.
\end{equation*}

Finally, primal ray definition \eqref{primalraydefmax} becomes
\begin{equation*}a^i\cdot\hat{x}=0\quad\forall i\in I \quad\text{and}\quad c\cdot \hat{x}>0,\end{equation*}
and dual ray definition \eqref{dualraydefmax} becomes
\begin{equation*}
     \sum_{i\in I} a^i\hat{y}_i =0 \quad\text{ and }\quad  \sum_{i \in I} \max\{l_i \hat{y}_i,u_i \hat{y}_i\}   <0.
\end{equation*}
Again, we can see that variables $w^i$ are also not needed to describe dual rays.

\subsection{Quadratic Objective}

Primal problem \eqref{primalproblem:max} for $I=\{i_0\}$ and $F^{i_0}$ and $g_{i_0}$ as in Proposition~\ref{simplequadraticobjlemma} reduces to 
\begin{equation*}
\operatorname{opt}_P:=\sup_{x\in \mathbb{R}^J}  c\cdot x - h_{{i}_0}(x)=\sup_{x\in \mathbb{R}^J}  c\cdot x - g_{i_0}\left(F^{i_0}(x)\right)
\end{equation*}
and the associated dual problem  \eqref{dualproblem:max} reduces to
\begin{subequations}\label{quadobjdual:max}
\begin{alignat}{3}
\operatorname{opt}_D:
=&\inf_{\substack{y_{i_0}\in \mathbb{R},\\ w^{i_0}\in \mathbb{R}^J}}\quad&   y_{i_0}h_{i_0}\left(w^{i_0}\right)\\
&s.t.&&\notag\\
&&y_{i_0} \nabla h\left(w^{i_0}\right)&= c,\quad y_{i_0}=1.
\label{quadobjdualfinal:max}
\end{alignat}
\end{subequations}
Because $F^i$ is not an affine function, variables $w^i$ do appear in \eqref{quadobjdual:max}. 

All primal solutions are feasible (as the primal problem is unconstrained linear plus quadratic minimization), dual feasibility is $\bar{y}_{i_0}=1$, and complementary slackness is 
\begin{equation}\label{quadobjcs:max}
    h_{i_0}(\bar{x})+
    \bar{y}_{i_0}h_{i_0}\left(\bar{w}^{i_0}\right)=\bar{y}_{i_0} \nabla h_{i_0}\left(\bar{w}^{i_0}\right)\cdot\bar{x}.
\end{equation}
In addition, primal ray definition \eqref{primalraydefmax} becomes
$\nabla h_{i_0}\left(\hat{x}\right)=0$  and $c\cdot \hat{x}>0$. Similarly,
 dual ray definition \eqref{dualraydefmax} becomes
 \begin{equation*}
     \hat{y}_{i_0}\nabla h\left(\hat{w}^{i_0}\right) =0,\quad  \hat{y}_{i_0} =0,\quad  \hat{y}_{i_0}h_{i_0}\left(\hat{w}^{i_0}\right)   <0.
\end{equation*}
In particular, we can check that a dual ray does not exist as expected from the problem being unconstrained.

Similarly to Section~\ref{quadraticobjsec} to get a \emph{clearly} convex re-formulation of the dual problem we can eliminate variable $y_{i_0}$ using the dual feasibility equation $y_{i_0}=1$ to obtain the more familiar dual problem (e.g. see \cite[(3)]{vandenberghe2010cvxopt} and \cite[Section 6.3]{roos2020universal}) given by
\begin{subequations}
    \begin{alignat}{3}
    \label{quadobjdualconvex:max}
\operatorname{opt}_D:=&\inf_{ w^{i_0}\in \mathbb{R}^J}\quad&   h_{i_0}\left(w^{i_0}\right)\\
&s.t.&&\notag\\
&&\nabla h_{i_0}\left(w^{i_0}\right) &= c
\end{alignat}
\end{subequations}
This reformulation does not need variable $y_{i_0}$ and all its dual solutions are module-wise feasible (i.e. dual feasibility reduces to the stationarity equation). In addition, complementary slackness conditions \eqref{quadobjcs:max} become
$h_{i_0}(\bar{x})+h_{i_0}\left(\bar{w}^{i_0}\right)= \nabla h_{i_0}\left(\bar{w}^{i_0}\right)\cdot\bar{x}$,
which, by Lemma~\ref{quadraticform:lemma}, is equivalent to
\begin{equation}\label{quadobjcsfinal:max}
   \left( \nabla h_{i_0}\left(\bar{x}\right)-\nabla h_{i_0}\left(\bar{w}^{i_0}\right)\right)\cdot\left(\bar{x}-\bar{w}^{i_0}\right)=0.
\end{equation}
In particular, complementary slackness condition \eqref{quadobjcsfinal:max} holds if $\bar{x}=\bar{w}^{i_0}$ in which case the stationarity conditions become the standard optimality conditions for linear minus quadratic unconstrained maximization given by $c=\nabla h_{i_0}\left(\bar{x}\right)$. However, the complementary slackness connection between $\bar{x}$ and $\bar{w}^{i_0}$ can also be achieved through $\nabla h_{i_0}\left(\bar{x}\right)=\nabla h_{i_0}\left(\bar{w}^{i_0}\right)$, which may be weaker than $\bar{x}=\bar{w}^{i_0}$ when $\nabla^2 h$ is singular.

\subsection{Quadratic Constraints}

Primal problem \eqref{primalproblem:max} for $I=\{i_0\}$ and $F^{i_0}$ and $g_{i_0}$ as in Proposition~\ref{quadraticconstrlemma} reduces to 
\begin{align*}
\operatorname{opt}_P:&=\begin{alignedat}[t]{3}&\sup_{x\in \mathbb{R}^J}\quad&   c\cdot x\\
&s.t.&&\\
&&h_{i_0}(x)+q^{i_0}\cdot x &\leq u_{i_0}&
\end{alignedat}
\end{align*}
and the associated dual problem  \eqref{dualproblem:max} reduces to
\begin{subequations}
\label{quadconstrdualtoref:max}
\begin{align}
\operatorname{opt}_D:
=&\inf_{{y_{i_0}\in \mathbb{R}, w^{i_0}\in \mathbb{R}^J}}    \left(h_{i_0}\left(w^{i_0}\right)+u_{i_0}\right)y_{i_0}\\
&s.t.\notag\\
&\quad\quad y_{i_0} \left(\nabla h_{i_0}\left(w^{i_0}\right)+q^{i_0}\right)= c,\quad y_{i_0}\geq 0.
\end{align}
\end{subequations}
Because $F^i$ is not an affine function, variables $w^i$ do appear in \eqref{quadconstrdualtoref:max}. 

Primal feasibility becomes  $h_{i_0}(\bar{x})+q^{i_0}\cdot \bar{x} \leq u_{i_0}$, dual feasibility is $\bar{y}_{i_0}\geq 0$ plus $y_{i_0} \left(\nabla h_{i_0}\left(w^{i_0}\right)+q^{i_0}\right)= c$  and and using \cref{quadraticform:lemma} we can write complementary slackness as 
\begin{align}
\notag
\bar{y}_{i_0} \left(h\left(\bar{w}^{i_0}\right)\right.&+\left.q^{i_0}\cdot \bar{w}^{i_0} -u_{i_0}\right)\\&=\bar{y}_{i_0}\left(\nabla h\left(\bar{w}^{i_0}\right)+q^{i_0}\right)\cdot\left(\bar{w}^{i_0}-\bar{x}\right).\label{quadraticconstcs:max}
\end{align}
In addition, primal ray definition \eqref{primalraydefmax} becomes
\begin{equation}
\nabla h_{i_0}\left(\hat{x}\right)=0, \quad q^{i_0}\cdot \hat{x}\leq 0 \quad\text{and}\quad c\cdot \hat{x}>0,\end{equation}
and dual ray definition \eqref{dualraydefmax} becomes
$\hat{y}_{i_0} \left(\nabla h_{i_0}\left(\hat{y}^{i_0}\right)+q^{i_0}\right)=0$,     $\hat{y}_{i_0} \geq 0$ and $\left(h_{i_0}\left(\hat{y}^{i_0}\right)+u_{i_0}\right)\hat{y}_{i_0}  <0$.

Finally, similar to Section~\ref{quadraticconstsec}, we can use  \cref{quadraticformulationlemma} to reformulate \eqref{quadconstrdualtoref:max}  as the convex optimization problem 
\begin{subequations}\label{quadconstrdualconv:max}
\begin{alignat}{3}
\operatorname{opt}_D:=&\inf_{\substack{y_{i_0}\in \mathbb{R},\\ z^{i_0}\in \mathbb{R}^J}}    \mathcal{G}_{h^{i_0}}\left(z^{i_0}, y_{i_0}\right)+u_{i_0}y_{i_0}\\
&s.t.\notag\\
&\quad\quad  \nabla h_{i_0}\left(z^{i_0}\right)+q^{i_0}y_{i_0}= c,\quad y_{i_0}\geq 0,
\end{alignat}
\end{subequations}
and we can convert between the solution for \eqref{quadconstrdualtoref:max} and \eqref{quadconstrdualconv:max} through 
\begin{equation*}
z^i\left(w^i,y_i\right):=y_i w^i,\quad
w^i\left(z^i,y_i\right):=\begin{cases}z^i/y_i &y_i>0 \\0&\text{o.w.}\end{cases}.
    \end{equation*}

\end{APPENDICES}
\end{document}